\documentclass[12pt]{amsart}
\usepackage{amsmath}
\usepackage{amsfonts}
\usepackage{amssymb}
\usepackage{amscd}
\usepackage{mathtools}
\usepackage{amsxtra}
\usepackage{amsthm}
\usepackage{graphicx}
\usepackage[abbrev,alphabetic]{amsrefs}
\RequirePackage[dvipsnames,usenames]{color}
\usepackage{soul,xcolor}
\setstcolor{red}
\usepackage{stmaryrd}
\usepackage{booktabs}
\usepackage{multirow}
\newtagform{tiny}{\tiny(}{)}

\usepackage{mathtools}
\usepackage{hyperref}
\usepackage[margin=1.25in]{geometry}
\usepackage{mathrsfs}

\usepackage{amsthm}
\usepackage{comment}
\usepackage[all,cmtip]{xy}
\usepackage{tikz-cd}
\usetikzlibrary{cd}

\tikzcdset{
  cells={font=\everymath\expandafter{\the\everymath\displaystyle}},
}

\usepackage[all]{xy}

\usepackage{cleveref}

\makeatletter
\def\@tocline#1#2#3#4#5#6#7{\relax
  \ifnum #1>\c@tocdepth 
  \else
    \par \addpenalty\@secpenalty\addvspace{#2}%
    \begingroup \hyphenpenalty\@M
    \@ifempty{#4}{%
      \@tempdima\csname r@tocindent\number#1\endcsname\relax
    }{%
      \@tempdima#4\relax
    }%
    \parindent\z@ \leftskip#3\relax \advance\leftskip\@tempdima\relax
    \rightskip\@pnumwidth plus4em \parfillskip-\@pnumwidth
    #5\leavevmode\hskip-\@tempdima
      \ifcase #1
       \or\or \hskip 1em \or \hskip 2em \else \hskip 3em \fi%
      #6\nobreak\relax
    \hfill\hbox to\@pnumwidth{\@tocpagenum{#7}}\par
    \nobreak
    \endgroup
  \fi}
\makeatother

\newcommand{\Z}{\mathbb{Z}}

\newcommand{\cB}{\mathcal{B}}

\newcommand{\perf}{\mathrm{perf}}

\makeatletter 
\DeclareRobustCommand{\graded}{%
    \@ifnextchar\bgroup{\graded@with}{\ensuremath{\mathrm{gr}}}%
}
\newcommand{\graded@with}[1]{\ensuremath{#1\mathrm{-gr}}}
\makeatother

\DeclareMathOperator{\grHom}{\underline{Hom}}

\newcommand{\comp}[2]{\Lambda_{#1}(#2)}
\newcommand{\dcomp}[2]{d\Lambda_{#1}(#2)}
\newcommand{\grcomp}[2]{\Lambda^{\graded}_{#1}(#2)}

\DeclareMathOperator{\grwedge}{{\wedge \graded}}

\def\var{\overline}

\DeclareMathOperator{\Supp}{Supp}

\DeclareMathOperator{\Hom}{Hom}

\DeclareMathOperator{\Ker}{Ker}

\DeclareMathOperator{\ev}{ev}

\DeclareMathOperator{\ppt}{ppt}

\theoremstyle{plain}
\newtheorem{theorem}{Theorem}[section]

\newtheorem{proposition}[theorem]{Proposition}

\newtheorem{lemma}[theorem]{Lemma}

\newtheorem{corollary}[theorem]{Corollary}

\newtheorem*{claim*}{Claim}

\newtheorem{theoremA}{Theorem}

\theoremstyle{definition}
\newtheorem{definition}[theorem]{Definition}

\newtheorem{example}[theorem]{Example}

\newtheorem{construction}[theorem]{Construction}

\newtheorem*{setup*}{Setup}

\theoremstyle{remark}
\newtheorem{remark}[theorem]{Remark}

\theoremstyle{plain}

\numberwithin{equation}{section}



\crefname{theorem}{Theorem}{Theorems}
\crefname{theoremA}{Theorem}{Theorems}
\crefname{proposition}{Proposition}{Propositions}
\crefname{lemma}{Lemma}{Lemmas}
\crefname{corollary}{Corollary}{Corollaries}
\crefname{conjecture}{Conjecture}{Conjectures}
\crefname{claim}{Claim}{Claims}
\crefname{notation}{Notation}{Notations}
\crefname{remark}{Remark}{Remarks}
\crefname{example}{Example}{Examples}
\crefname{definition}{Definition}{Definitions}
\crefname{construction}{Construction}{Constructions}

\usepackage{mymacros_common_Ryo, mymacros_english_paper_Ryo}

\title{Graded perfectoid rings}
\author{Ryo Ishizuka}
\address{Institute of Science Tokyo, Tokyo 152-8551, Japan}
\email{ishizuka.r.ac@m.titech.ac.jp}

\author{Shou Yoshikawa}
\address{Institute of Science Tokyo, Tokyo 152-8551, Japan}
\email{yoshikawa.s.9fe9@m.isct.ac.jp}



\begin{document}

\begin{abstract}
We introduce and study graded perfectoid rings as graded analogues of Scholze’s (integral) perfectoid rings.  
We establish a categorical equivalence between graded perfectoid rings and graded perfect prisms, extending the Bhatt--Scholze's correspondence to the graded setting.
We also construct the initial graded perfectoid cover of any graded semiperfectoid rings and prove a graded version of Andr\'e's flatness lemma.
These results lay the foundations for a graded theory of perfectoid rings.
\end{abstract}

\maketitle

\tableofcontents

\setcounter{tocdepth}{1}

\section{Introduction}

Let $G$ be a torsion free abelian group and let \(p\) be a prime number.
Perfectoid rings, introduced by Scholze in \cite{scholze2012Perfectoida}, play a central role in modern $p$-adic geometry, for example, the theory of prismatic cohomology developed by Bhatt and Scholze~\cite{bhatt2022Prismsa}.
Beyond that, after solving the direct summand conjecture \cites{andre2018Conjecture}, perfectoid methods have already become a fundamental tool in mixed characteristic commutative ring theory and singularity theory, for example \cites{ma2021Singularities,bhatt2024Perfectoid}. Such methods mainly focus on the usage of perfectoid rings as a nice cover of Noetherian local rings.

On the other hand, graded rings play an important role both in commutative algebra and algebraic geometry, since their algebraic properties are closely related to global geometric properties of projective varieties (see, for example, \cites{goto1978Graded}).
It is natural to ask that how to apply perfectoid methods in such global geometry.
However, except in trivial cases, a graded ring is never \(p\)-complete and hence cannot be perfectoid.  
Thus, when one wishes to study global conditions in mixed characteristic, it is natural to consider a graded analogue of perfectoid rings.

In this paper, we introduce and study \emph{$G$-graded perfectoid rings}, which is a desired analogue we thought.
The goal of this work is to extend the foundational results of perfectoid theory to the graded setting, especially, to establish categorical equivalences with graded perfect prisms.

Our starting point is the following definition of graded analogues of the basic notions related to perfectoid rings.
Here, a \emph{\(G\)-graded ring} \(R\) is a commutative ring \(R\) equipped with the decomposition \(R = \bigoplus_{g \in G} R_g\) of abelian groups such that \(R_g R_{h} \subseteq R_{g + h}\).

\begin{definition}[{\Cref{DefGradedPerfectoid} and \Cref{graded-perfect-prism}}]
Let $R$ be a $G$-graded ring.
\begin{itemize}
    \item We say that $R$ is \emph{graded perfectoid} if $R$ is $p$-adically gradedwise complete,\footnote{For an element \(x \in R_0\), we say that \(R\) is \emph{\(x\)-adically gradedwise complete} if each \(R_0\)-module \(R_g\) is \(x\)-adically complete for all \(g \in G\). See \Cref{DefTopGraded} for more general definition.} and the \(p\)-adic completion $R^{\wedge p}$ of \(R\) is perfectoid.
    \item We say that $R$ is \emph{graded semiperfectoid} if there exist a $G$-graded perfectoid ring $P$ and a graded surjective ring homomorphism $P \to R$.
    \item We say that $(R,I)$ is a \emph{graded perfect prism} if $R$ is a $\delta$-ring with $\delta(R_g) \subseteq R_{pg}$, $I$ is a homogeneous ideal, and $(R^{\wedge (p,I)},I)$ is a perfect prism, where \(R^{\wedge(p, I)}\) is the \((p, I)\)-adic completion of \(R\).
\end{itemize}
\end{definition}

Our first main result shows that the graded perfectoid condition can be expressed in several equivalent forms:

\begin{theoremA}[\cref{except-for-p-comp,CatEquivPerfdPrism}, cf.~\cite{gabber2018Foundations}*{Proposition~16.6.1}]\label{intro:equiv-perfd}
Let $R$ be a $p$-adically gradedwise complete $G$-graded ring.
Then the following are equivalent:
\begin{enumerate}
    \item This $R$ is a graded perfectoid ring.
    \item The \(p\)-adic completion \(R^{\wedge p}\) and the degree-\(0\) part \(R_0\) are perfectoid;
    \item There exists a topological nilpotent element \(\varpi \in R_0\)  with \(p \in \varpi^pR_0\), the Frobenius map
    \[
        R/\varpi R \xrightarrow{\,a \mapsto a^p\,} R/\varpi^p R
    \]
    is an isomorphism, and the multiplicative map
    \[
        R[\varpi^\infty] \xrightarrow{\,a \mapsto a^p\,} R[\varpi^\infty]
    \]
    is bijective. 
    \item There exists a $G$-graded perfect prism $(A,I)$ such that $R \simeq A/I$ as graded rings.
\end{enumerate}
\end{theoremA}

The third condition is base on the similar one for usual perfectoid rings given by \cites{ishiro2025Perfectoida, gabber2018Foundations} (see \Cref{EquivPerfectoid}). This does not require any Witt/tilt construction whose graded variant (\Cref{graded-witt-tilt}) is a little bit complicated.

On the forth one, we have to use such constructions and, more generally, we have the following theorem generalizing the fundamental equivalence between perfectoid rings and perfect prisms to the graded context:

\begin{theoremA}[\Cref{cat-equiv}]\label{intro:cat-equiv}
The following two categories are equivalent:
\begin{itemize}
    \item the category of \(G\)-graded perfectoid rings $R$;
    \item the category of \(G\)-graded perfect prisms \((A,I)\).
\end{itemize}
This categorical equivalence is compatible with the usual one in \cite[Theorem 3.3]{bhatt2022Prismsa} through the completion functors \(R \mapsto R^{\wedge p}\) and \((A, I) \mapsto (A^{\wedge(p, I)}, I)\).
\end{theoremA}

Next task is to prove the existence of the initial object representing the perfectoidization functor in the graded setting:

\begin{theoremA}[\cref{gr-perfdion}]\label{intro:gr-perfdion}
Let \(R\) be a \(G\)-graded semiperfectoid ring.
Then there exists a \(G\)-graded perfectoid ring \(R_{\perfd,{\gr}}\) over \(R\),
which is the initial \(G\)-graded perfectoid over \(R\) and its \(p\)-adic completion coincides with the perfectoidization \(R_{\perfd}\) of \(R\) in the sense of \cite{bhatt2022Prismsa}.  
\end{theoremA}

We give two applications of this existence of graded perfectoidization.
First, we can establish a graded analogue of Andr\'e’s flatness lemma, which ensures the existence of compatible systems of $p$-power roots under $p$-completely faithfully flat extensions:

\begin{theoremA}[\Cref{Andre-graded}]\label{intro:Andre-graded}
Let \(R\) be a \(G\)-graded perfectoid ring and let \(\mcalS\) be a set of homogeneous elements in some graded integral extension of \(R\).
Then there exists a $p$-completely faithfully flat map $R \to R'$ of \(G\)-graded perfectoid rings such that \(R'\) admits compatible systems of \(p\)-power roots of elements in \(\mcalS\).
\end{theoremA}

Second, we can show that the equivalence of perfectoid purity introduced in \cite{bhatt2024Perfectoid}.

\begin{theoremA}[{\Cref{GradedPerfdPure} and \Cref{gr-p-pure}}] \label{intro:gr-p-pure}
    Let \(R\) be a \(p\)-torsion-free \(G\)-graded Noetherian ring.
    Assume that $R$ has a unique graded maximal ideal $\mfrakm$ with $p \in \mfrakm$.
    Then the following are equivalent:
    \begin{enumerate}
        \item This ring \(R\) is graded perfectoid pure in the sense that there exists a graded ring homomorphism \(R \to P\) to a graded perfectoid ring \(P\) such that pure as a ring homomorphism.
        \item The \(p\)-adic completion \(R^{\wedge p}\) is perfectoid pure.
        \item The localization \(R_{\mfrakm}\) is perfectoid pure.
    \end{enumerate}
    Moreover, in this case we have $\ppt(R^{\wedge p},p) = \ppt(R_\mfrakm,p)$.
\end{theoremA}

As a technical tool for proving the main theorems of this paper,  
we introduce the notion of \emph{pro-\(G\)-graded rings} (\Cref{def-pro-gra}).  
Roughly speaking, a pro-\(G\)-graded ring is a pair \((R, R_{\gr})\) consisting of a linear topological ring \(R\) and its graded part \(R_{\gr}\),  
where \(R\) is the completion of \(R_{\gr}\) with respect to its homogeneous linear topology.  
This framework allows us to treat graded and topological structures simultaneously, and serves as a bridge between gradedwise complete rings and their completions in the perfectoid setting.
Actually, the above theorems only use \(p\)-adic topology but in general we can prove more general results on \emph{pro-\(G\)-graded perfectoid rings} and \emph{pro-\(G\)-graded perfect prisms}, namely, the above theorems hold in ``adic'' situations outside \(p\)-adic ones.

\begin{remark}
Prior to this work, Gabber and Ramero conducted a similar study on graded rings and perfectoid rings in \cite[\S 16.6]{gabber2018Foundations}. In particular, the equivalences between (1)--(3) in \Cref{intro:equiv-perfd} follow from their work; see \Cref{RemarkGRGradedRing} and \Cref{RelationGR} for a detailed comparison.
\end{remark}

\subsection*{Structure of this paper}
This paper is organized as follows.  
In Section~2, we recall basic definitions and properties of perfectoid rings and perfectoidizations.  
Section~3 introduces pro-\(G\)-graded rings and the notion of gradedwise completion, providing categorical equivalences between graded topological and pro-graded settings.  
In Section~4, we develop the theory of graded perfectoid rings and prove the results above, including the graded André’s flatness lemma.  
Finally, Section~5 treats graded perfect prisms and establishes the categorical equivalence between graded perfectoid rings and graded perfect prisms.

\subsection*{Acknowledgments}
The authors would like to thank Jack J. Garzella and Joe Waldron for their valuable discussions about global variant of perfectoid theory and L\'eo Navarro Chafloque for his many insights on perfectoidization.
This work was started while we attended the conference `\(p\)-adic and Characteristic \(p\) Methods in Algebraic Geometry' at EPFL and `the Summer Research Institute in Algebraic Geometry' at Colorado State University in 2025. We are very grateful for these opportunity and their hospitality.

The first-named author was supported by JSPS KAKENHI Grant number 24KJ1085.
The second-named author was supported by JSPS KAKENHI Grant number JP24K16889.

\section{Preliminaries}

\subsection{Notation and terminologies}
\begin{enumalphp}
    \item We fix a prime number \(p\).
    \item We only consider commutative rings with unit which is often referred to as \emph{rings} simply.
    \item Let $R$ be a ring, $M$ an $R$-module, and $I$ an ideal of $R$.
    We define the \emph{$I$-torsion part of $M$} (resp., \emph{\(I^{\infty}\)-torsion part of \(M\)}) by
    \[
    M[I]\defeq\{m \in M \mid Im=0\} \quad \text{(resp., \(M[I^{\infty}] \defeq \set{m \in M}{\exists n > 0, \  I^nm = 0}\))}
    \]
    \item A topological ring has a \emph{linear topology} if there exists a fundamental system of neighborhoods of \(0\) consisting of ideals. A linear topological ring \(R\) is called \emph{complete} if any Cauchy net in \(R\) uniquely converges. If there exists the universal complete topological \(R\)-algebra, then it is called the \emph{completion} of \(R\).
    \item If a topological ring \(R\) has an ideal \(I\) such that \(I\) is finitely generated and the set \(\{I^n\}_{n \geq 0}\) consists a fundamental system of neighborhoods of \(0\), then we call \(R\) as an \emph{adic ring} and \(I\) as an \emph{ideal of definition} of \(R\).
    Furthermore, morphisms between adic rings are defined to be continuous ring homomorphisms.
    \item Denote \(G\) to be a torsion-free abelian group with identity element \(0\).
    \item A \emph{\(G\)-graded ring} is a pair of a commutative ring \(R\) and additive subgroups \(\{R_g\}_{g \in G}\) such that \(R\) is decomposed to the coproduct \(R = \bigoplus_{g \in G} R_g\) and satisfies \(R_g R_{g'} \subseteq R_{g + g'}\) for any \(g, g' \in G\). A \emph{homogeneous element} of \(R\) is an element in \(R_g\) for some \(g \in G\) and a \emph{homogeneous ideal} is an ideal of \(R\) which is generated by homogeneous elements.
    The subset $\{g \in G \mid R_g \neq 0\}$ of $G$ is called by the \emph{support of $R$} and denoted by $\Supp{R}$.
    Furthermore, for a submonoid $H$ of $G$, if the support of $R$ is contained in $H$, then we say that $R$ is \emph{$H$-graded ring}.
    \item A \emph{\(G\)-graded ring morphism} between \(G\)-graded rings \(R = \bigoplus_{g \in G} R_g\) and \(S = \bigoplus_{g \in G} S_g\) is a ring homomorphism \(\varphi \colon R \to S\) such that \(\varphi(R_g) \subseteq S_g\) for any \(g \in G\).

    \item Let $R$ be a graded ring.
    For each \(r \in \mathbb{Q}\), if $rG \subseteq G$, then we define a \(G[1/r]\)-graded ring \(R^{[r]}\) by
    \[
      (R^{[r]})_g \defeq R_{r g},
    \]
    the degree-\(r g\) part of the graded ring \(R\).
    \item To emphasize the meaning of completions, we will use several notations of them. The symbol \(\widehat{(-)}\), \((-)^{\wedge I}\) and \(\comp{I}{-}\) are only used to denote the (classical) (\(I\)-)adic completion of rings and modules.
    \item The symbol \(\dcomp{I}{-}\) is used to denote the derived (\(I\)-)completion of rings and modules. In general, this is a functor on a derived category of modules.
\end{enumalphp}

\subsection{Perfectoid rings}
We freely use the notion of perfectoid rings, prisms, and related concepts, e.g., perfectoidization, following \cites{bhatt2018Integral, bhatt2022Prismsa}.
First we start a perfectoid property of (not necessarily \(p\)-adic topological) rings:

\begin{definition}[{\cite[Definition 4.4]{takaya2025Relative}}]
    An adic ring \(R\) is called a \emph{perfectoid ring} if the underlying ring \(R\) is a perfectoid ring.
    When we say `a ring \(R\) is perfectoid', then it means \(R\) is a perfectoid ring as \(p\)-adic rings.
\end{definition}

\begin{lemma} \label{CompatibleSystemModp}
    Let \(R\) be a \(p\)-adically complete ring such that \(R/pR\) is semiperfect.
    Take an element \(f\) of \(R\).
    Then there exists an element \(g\) of \(R\) such that \(f \equiv g \bmod pR\) and \(g\) has a compatible system \(\{g^{1/p^n}\}_{n \geq 0}\) of \(p\)-power roots of \(g\) in \(R\).
\end{lemma}

\begin{proof}
    Since the Frobenius map on \(R/pR\) is surjective, there exists a sequence of elements \((f_0, f_1, \dots)\) of \(R\) such that
    \begin{equation*}
        f_0 = f \quad \text{and} \quad f_{i+1}^p \equiv f_i \bmod pR.
    \end{equation*}
    This condition yields \(p\deg(f_{i+1}) = \deg(f_i)\) and, for any \(n \geq 0\), this gives a Cauchy sequence \((f_i^{p^{i-n}})_{i \geq n}\) of \(R\) with respect to the \(p\)-adic topology.
    Since \(R\) is \(p\)-adically complete, this Cauchy sequence has the limit \(g_n \defeq \lim_{i \to \infty} f_i^{p^{i-n}}\) for each \(n \geq 0\).
    By construction, \(g_n\) is an element of \(R\) and the equality \(g_{n+1}^p = g_n\) holds in \(R\).
    Moreover, since the Cauchy sequence \((f_i^{p^{i}})_{i \geq 0}\) becomes a constant sequence \((\overline{f})_{i \geq 0}\) after taking modulo \(p\), we can show that \(g_0 \equiv f \bmod pR\).
    Taking \(g^{1/p^n} \defeq g_n\), we have the desired element \(g\).
\end{proof}

\begin{lemma} \label{BoundedTorsionPerfd}
    Let \(R\) be a perfectoid ring and let \(I\) be a finitely generated ideal of \(R\) containing \(p\).
    Then the following assertions hold.
    \begin{enumerate}
        \item There exists a generator \(f_1, \dots, f_r\) of \(I\) such that \(f_i\) has a compatible system of \(p\)-power roots in \(R\).
        \item  The ring \(R\) has bounded \(f_i^{\infty}\)-torsion for such choice of \(f_i\) in (1). In particular, \(R\) has bounded \(I^{\infty}\)-torsion.
        \item The \(I\)-adic completion of \(R\) is a perfectoid ring and agrees with the derived \(I\)-completion \(\dcomp{I}{R}\) of \(R\).
    \end{enumerate}
\end{lemma}

\begin{proof}
    (1): This follows from \Cref{CompatibleSystemModp}.

    (2): It suffices to show that \(R\) has bounded \(f_i^\infty\)-torsion for any \(i = 1, \dots, r\) where \(f_i\) are in (1).
    If \(x \in R\) satisfies \(f_i^Nx = 0\) for some \(N > 0\), we can show that \((f_i^{N/p^n}x)^{p^n} = 0\) in \(R\) for any \(n > 0\).
    Taking large enough \(n\) and using the reducedness of perfectoid rings, we can conclude \(f_ix = 0\) in \(R\). This shows the boundedness.

    (3): This is \cite[Proposition 2.1.11 (e)]{CS2024Purity}. For the first statement, see also \cite[Corollary 4.11]{takaya2025Relative}.
\end{proof}

We record an equivalent condition for perfectoid rings following \cite[Theorem 3.50]{ishiro2025Perfectoida}, which is based on \cite[Theorem 16.3.62 and Corollary 16.3.73]{gabber2018Foundations}.

\begin{theorem}[{\cite[Theorem 3.50]{ishiro2025Perfectoida}}] \label{EquivPerfectoid}
    A ring \(R\) is a perfectoid ring if and only if there exists an element \(\varpi \in R\) satisfies the following:
    \begin{enumerate}
        \item \(R\) is \(\varpi\)-adically complete and \(p \in \varpi^pR\).
        \item The \(p\)-th power map \(R/\varpi R \xrightarrow{a \mapsto a^p} R/\varpi^p R\) is an isomorphism of rings.
        \item The multiplicative map
            \begin{align*}
                R[\varpi^\infty] & \longrightarrow R[\varpi^\infty] \\
                x & \longmapsto x^p
            \end{align*}
            is bijective.
    \end{enumerate}
\end{theorem}

\begin{remark}[{\cite[Remark 3.51]{ishiro2025Perfectoida}}] \label{ChoiceVarpi}
    Take an arbitrary element \(\pi\) of a perfectoid ring \(R\).
    If we assume that \(R\) is \(\pi\)-adically complete and \(p \in \pi^pR\), then \(R/\pi R \xrightarrow{a \mapsto a^p} R/\pi^p R\) is an isomorphism of \cite[Lemma 3.10(i)]{bhatt2018Integral}.
    Moreover, the proof of \Cref{EquivPerfectoid} in \cite{ishiro2025Perfectoida} shows that the multiplicative map \(R[\pi^{\infty}] \xrightarrow{a \mapsto a^p} R[\pi^\infty]\) becomes bijective.
\end{remark}

One of the advantages of this condition is that that condition does not require the tilt, the ring of Witt vectors, and the Fontaine's theta map of which the graded variants need more works as \Cref{pro-graded-tilt}, \Cref{pro-graded-Witt}, and \Cref{SharpMapThetaMap}.
By using this, we will be able to give some equivalent conditions on graded perfectoid rings (\Cref{graded-perfectoid-equiv}).

\begin{lemma}\label{torsion-compare}
Let \(R\) be a ring and let \(\varpi \in R\).
Consider the maps
\[
\begin{aligned}
\pi_\infty &\colon R[\varpi^\infty] \longrightarrow R[\varpi^\infty],
& a &\longmapsto a^p, \\
\pi &\colon R[\varpi^p] \longrightarrow R[\varpi^p],
& a &\longmapsto a^p.
\end{aligned}
\]
Then \(\pi_\infty\) is bijective if and only if so is \(\pi\).
Moreover, in either case, we have \(R[\varpi] = R[\varpi^\infty]\).
\end{lemma}

\begin{proof}
First, assume that \(\pi_\infty\) is bijective.
To show \(R[\varpi] = R[\varpi^\infty]\), take \(x \in R\) such that \(\varpi^{p^n} x = 0\) for some \(n \ge 1\).
By the surjectivity of \(\pi_\infty\), there exists \(y \in R[\varpi^\infty]\) with \(x = y^{p^n}\).
Then
\[
(\varpi y)^{p^n} = \varpi^{p^n} x = 0.
\]
By the injectivity of \(\pi_\infty\), it follows that \(\varpi y = 0\).
Hence \(\varpi x = \varpi y^{p^n} = 0\), as desired.
In particular, \(R[\varpi^p] = R[\varpi^\infty]\), and \(\pi\) is bijective.

\smallskip
Conversely, assume that \(\pi\) is bijective.
It suffices to show \(R[\varpi^p] = R[\varpi^\infty]\).
We claim that if \(\varpi^{p+1} x = 0\), then \(\varpi^p x = 0\) for any \(x \in R\).
By surjectivity of \(\pi\), there exists \(y \in R[\varpi^p]\) such that \(y^p = \varpi x\).
Then
\[
(\varpi y)^p = \varpi^{p+1} x = 0.
\]
By the injectivity of \(\pi\), we have \(\varpi y = 0\),
and hence \(\varpi^2 x = \varpi y^p = 0\).
Since \(2 \le p\), this implies \(\varpi^p x = 0\), as required.

\smallskip
Therefore, in both cases, \(\pi_\infty\) is bijective, and we obtain
\(R[\varpi] = R[\varpi^\infty]\), completing the proof.
\end{proof}

The following decomposition property of the perfectoidization is obtained in the on-going work of the first-named author with L\'eo Navarro Chafloque.
It reduces the computation of perfectoidizations to the one of the \(p\)-torsion-free part and the characteristic part.
The proof is not given in this paper, but for later use, the claim is stated here.

\begin{proposition} \label{DecompositionPerfectoidization}
    Let $R$ be a ring and suppose that there exists an initial map to a perfectoid ring $R\to R_{\perfd}$. Denote $R' \defeq R/R[p^{\infty}]$.
    Then the universal perfectoid \(R'\)-algebra \((R')_{\perfd}\) exists and the natural map
    \begin{equation*}
        R_{\perfd}\to (R')_{\perfd}\times_{((R')_{\perfd}/p(R')_{\perfd})_{\perf}}(R/pR)_{\perf}
    \end{equation*}
    is an isomorphism, where the symbol \((-)_{\perf}\) denotes the perfection of \(\setF_p\)-algebras.
\end{proposition}

\subsection{Graded homomorphisms}

\begin{lemma} \label{EquivPureMap}
    Let \(R\) be a \(G\)-graded ring.
    Take a graded \(R\)-module homomorphism \(\varphi \colon M \to N\).
    Then the following are equivalent.
    \begin{enumerate}
        \item This \(\varphi\) is pure as a (non-graded) \(R\)-module homomorphism.
        \item The base change \(M \otimes_R L \to N \otimes_R L\) is injective for all graded \(R\)-module \(L\).
    \end{enumerate}
    Therefore, we do not need to distinguish these two purity of graded morphisms.
\end{lemma}

\begin{proof}
    \((1) \Rightarrow (2)\): This is clear.

    \((2) \Rightarrow (1)\): Let \(L\) be a (non-graded) \(R\)-module.
    Take the \(G\)-graded abelian group
    \begin{equation*}
        L[T^G] \defeq \bigoplus_{g \in G} L T^g,
    \end{equation*}
    where \(T\) is a variable with the natural \(G\)-graded \(R\)-module structure.
    By assumption, the base change
    \begin{equation*}
        M \otimes_R L[T^G] \to N \otimes_R L[T^G]
    \end{equation*}
    is injective.
    Using the splitting injection \(M \otimes_R L \hookrightarrow M \otimes_R L[T^G]\), we can conclude that \(M \otimes_R L \to N \otimes_R L\) is injective.
\end{proof}

\begin{proposition}\label{graded-hom-hom}
Let $R$ be a $G$-graded ring and $M$ be a finitely presented graded $R$-module.
Then for every graded $R$-module $N$, the natural homomorphism
\[
\grHom_R(M,N) \to \Hom_R(M,N)
\]
is an isomorphism as \(R\)-modules, where \(\grHom_R(M, N)\) is the \(G\)-graded module of \(R\)-module homomorphisms used in \cite{goto1978Graded} and \citeSta{00JL}.
\end{proposition}

\begin{proof}
By definition, we have the natural injection $\varphi_M \colon \grHom_R(M,N) \to \Hom_R(M,N)$.
Taking a presentation of $M$ as a graded $R$-module
\[
\bigoplus_{i=1}^{s} R(m_i) \to \bigoplus_{j=1}^t R(n_j) \to M \to 0,
\]
we obtain the commutative diagram in which each horizontal sequence is exact
\[
\begin{tikzcd}
    0 \arrow[r] & \grHom_R(M,N) \arrow[r] \arrow[d,"\varphi_M"] & \grHom_R(\bigoplus_{j=1}^t R(n_j),N) \arrow[r] \arrow[d,"\varphi_1"] & \grHom_R(\bigoplus_{i=1}^s R(m_i),N) \arrow[d,"\varphi_2"] \\
    0 \arrow[r] & \Hom_R(M,N) \arrow[r]  & \Hom_R(\bigoplus_{j=1}^t R(n_j),N) \arrow[r]  & \Hom_R(\bigoplus_{i=1}^s R(m_i),N). 
\end{tikzcd}
\]
Since $\varphi_1$ and $\varphi_2$ are isomorphisms, so is $\varphi_M$, as desired.
\end{proof}

\begin{lemma}\label{purity-graded-local}
Let $R$ be a $G$-graded Noetherian ring.
Assume that the graded Jacobson radical\footnote{For any \(G\)-graded ring \(R\), the \emph{graded Jacobson radical} of \(R\) is the intersection of all graded maximal ideal of \(R\).} $R$ contains \(p\).
Let $M$ be a \(G\)-graded $R$-module which has bounded \(p^{\infty}\)-torsion.
Take a graded $R$-module homomorphism $\psi \colon R \to M$.
Then the following are equivalent.
\begin{enumerate}
    \item This \(\psi\) is pure.
    \item The \(p\)-adic completion $\psi^{\wedge p} \colon R^{\wedge p} \to M^{\wedge p}$ is pure
    \item The induced homomorphism \(\psi'_{\mfrakm} \colon R_\mfrakm \longrightarrow (M_{\mfrakm})^{\wedge p}\) is pure for any graded maximal ideal \(\mfrakm\) of \(R\).
    \item The localization \(\psi_{\mfrakm} \colon R_{\mfrakm} \to M_{\mfrakm}\) is pure for any graded maximal ideal \(\mfrakm\) of \(R\).
\end{enumerate}
\end{lemma}

\begin{proof}
    \((1) \Rightarrow (2)\): The \(p\)-adic completion \(R^{\wedge p} \to M^{\wedge p}\) now coincides with the derived \(p\)-completion since \(R\) and \(M\) have bounded \(p^\infty\)-torsion.
    Then the purity of \(\psi\) implies \(\psi^{\wedge p}\) is a pure homomorphism by \cite[Lemma 2.13]{bhatt2024Perfectoid}.

    \((2) \Rightarrow (3)\): Take any graded maximal ideal \(\mfrakm\) of \(R\).
    The purity of \(\psi^{\wedge p}\) implies the one of the localization
    \begin{equation*}
        (\psi^{\wedge p})_{\mfrakm} \colon (R^{\wedge p})_{\mfrakm} \to (M^{\wedge p})_{\mfrakm}.
    \end{equation*}
    Using \cite[Lemma 2.13]{bhatt2024Perfectoid} again, we can deduce that its \(p\)-adic completion is pure.
    This is the same as the purity of
    \begin{equation*}
        \psi_{\mfrakm}^{\wedge p} \colon R_{\mfrakm}^{\wedge p} \to M_{\mfrakm}^{\wedge p}.
    \end{equation*}
    Since \(R_{\mfrakm} \to R_{\mfrakm}^{\wedge p}\) is pure, the composition \(\psi'_{\mfrakm}\) is pure.

    \((3) \Rightarrow (4)\): Since \(\psi'_{\mfrakm}\) factors through \(\psi_{\mfrakm}\), the purity of \(\psi'_{\mfrakm}\) implies that of \(\psi_{\mfrakm}\) for any graded maximal ideal $\mfrakm$ of $R$.

    \((4) \Rightarrow (1)\): Writing $M$ as a filtered colimit of finitely generated graded $R$-submodules, we may assume that $M$ is finitely generated.

    We want to show that \(\psi\) splits.
    This is equivalent to saying that the evaluation morphism along the homogeneous element \(\psi(1)\) from the graded \(R\)-module of \(R\)-module homomorphisms
    \[
    e = \ev_{\psi(1)} \colon \grHom_R(M,R) \to R
    \]
    is surjective, where $\grHom_R(M,R)$ has a natural \(G\)-graded \(R\)-module structure as in \citeSta{00JL}.
    Assume the converse. Then the image \(\Image(e)\), which is a graded submodule of \(R\), is contained in a graded maximal ideal \(\mfrakm\) of \(R\).
    Since $\psi_{\mfrakm}$ is assumed to be pure and then $\psi_{\mfrakm}$ splits.
    Then the localization \(e_{\mfrakm}\) is surjective by \cref{graded-hom-hom}, which implies that \(\Image(e_{\mfrakm}) = \Image(e)_{\mfrakm} = R_{\mfrakm}\).
    However, this contradicts the containment \(\Image(e)_{\mfrakm} \subseteq \mfrakm R_{\mfrakm} \subsetneq R_{\mfrakm}\).
    So \(e\) is surjective.
    \qedhere

    
\end{proof}

\section{Pro-graded rings}

\subsection{Graded topological rings}

\begin{definition} \label{DefTopGraded}
We first fix a notion of graded topological rings and its related concepts:
\begin{itemize}
    \item In this article, a \emph{$G$-graded topological ring} $R$ is a $G$-graded ring $R=\bigoplus_{g\in G} R_g$ equipped with a homogeneous linear topology; that is, there exists a fundamental system of neighborhoods of $0$ consisting of homogeneous ideals. In particular, the set of all open homogeneous ideals is such a system, called \emph{the system of open homogeneous ideals} and denoted by $\mathcal{B}_{R}$.
    \item Let $R$ be a $G$-graded topological ring. If $R$ has a finitely generated homogeneous ideal $J$ such that $\{J^n\}_{n\ge 1}$ is a fundamental system of neighborhoods of $0$, then we call $R$ a \emph{$G$-graded adic ring}, and such $J$ a \emph{homogeneous ideal of definition} of $R$.
    \item We say that a $G$-graded topological ring $R$ is \emph{gradedwise complete} if $R_g$ is complete with respect to the topology induced from the inclusion $R_g\to R$ for every $g\in G$.
    \item A morphism of $G$-graded topological rings is a continuous $G$-graded ring homomorphism $\varphi\colon R\to S$.
\end{itemize}
\end{definition}

\begin{remark}\label{remk:intersection}
Let $R$ be a $G$-graded topological ring. If $I,J\in \mathcal{B}_R$, then $I\cap J$ is an open homogeneous ideal, hence $I\cap J\in \mathcal{B}_R$. Therefore $\mathcal{B}_R$ is closed under finite intersections and forms a cofiltered set under inclusion. For each $g\in G$, the completion of $R_g$ with respect to the induced topology is $\varprojlim_{I\in \mathcal{B}_R} R_g/(I\cap R_g)$, and the completion of $R$ is $\varprojlim_{I\in \mathcal{B}_R} R/I$.
\end{remark}

\begin{construction} \label{GradedwiseCompletion}
The inclusion functor from the category of gradedwise complete $G$-graded topological rings to the category of $G$-graded topological rings admits a left adjoint, called \emph{gradedwise completion}. Explicitly, for a $G$-graded topological ring $R$ we set
\[
R^{\grwedge} \defeq \grcomp{I}{R} \defeq \bigoplus_{g\in G}\ \varprojlim_{I\in \mathcal{B}_R} R_g/(I\cap R_g).
\]
Note that this need not coincide with the (usual) completion $\varprojlim_{I\in \mathcal{B}_R} R/I$ in the category of topological rings.
If there is no confusion of what the topology is, we just use the symbol \((-)^{\grwedge}\). However, when we will take some gradedwise completions with respect to different topologies, the symbol \(\grcomp{I}{-}\) will be used.
\end{construction}

For gradedwise completeness, the topological Nakayama lemma also holds (see also \Cref{SurjectiveProGraded}):

\begin{proposition}\label{graded-top-nakayama}
Let $\varphi\colon S\to R$ be a morphism of $G$-graded topological rings which is an open map. Assume $S$ is gradedwise complete and $R$ is separated. If for every $I\in \mathcal{B}_S$ the induced map $\varphi_I\colon S/I\twoheadrightarrow R/IR$ is surjective, then $\varphi$ is surjective.
\end{proposition}

\begin{proof}
Let $b\in R$ with homogeneous decomposition $b=\sum_{g\in G} b_g$. For each $I\in \mathcal{B}_S$ and $g\in G$, set $b_{I,g}=(R\to R/IR)(b_g)$ and choose $a_{I,g}\in S$ with $\varphi_I(a_{I,g})=b_{I,g}$. Then $(a_{I,g})_{I\in \mathcal{B}_S}$ defines an element of $\varprojlim_{I\in \mathcal{B}_S} S_g/(S_g\cap I)$, hence of $S_g$ by gradedwise completeness; denote it by $a_g$. Put $a\defeq\sum_{g\in G} a_g\in S$. Then $b-\varphi(a)\in IR$ for all $I\in \mathcal{B}_S$. Since $\varphi$ is open and $R$ is separated, we get $b=\varphi(a)$, proving surjectivity.
\end{proof}

The completion of graded topological rings usually do not have a grading except for the trivial case.
However, elements of the completion have a possibly infinitely many homogeneous decomposition, which we will call the \emph{topological homogeneous decomposition} in \Cref{defn-homog-decomp}.
We first prove the following description of the completion of graded topological rings:

\begin{proposition}\label{const-comp}
Let $R$ be a $G$-graded topological ring which is gradedwise complete. For $\alpha=(a_g)\in \prod_{g\in G} R_g$ and every homogeneous ideal $I\subset R$, set
\[
G_I^{\alpha}\defeq\{\,g\in G\mid a_g\notin I\,\}.
\]
Let
\[
\mathcal{H}\defeq\left\{\alpha=(a_g)\in \prod_{g\in G} R_g \ \middle|\ \text{for all } I\in \mathcal{B}_R,\ G_I^\alpha \text{ is finite}\right\}.
\]
Define $S\defeq\{(a_g)\in \mathcal{H}\}$ with componentwise addition and multiplication
\[
(a_g)\cdot(b_g)\defeq\Bigl(\ \sum_{l\in G} a_l\, b_{g-l}\ \Bigr).
\]
Then $S$ is the completion $\widehat{R}$ of $R$, and the natural map $\iota\colon R\to S$ is given by $\iota(a)=(a_g)$ for the homogeneous decomposition $a=\sum_{g} a_g$.
\end{proposition}

\begin{proof}
Let $\pi_I\colon R\to R/I$ be the projection. If $\alpha=(a_g),\beta=(b_g)\in \mathcal{H}$, then for each $g$ and $I$ the set $G'_g\defeq\{\,l\in G\mid a_l b_{g-l}\notin I\,\}=G_I^\alpha\cap(g-G_I^\beta)$ is finite, so $\bigl(\sum_{l\in G'_g}\pi_I(a_l b_{g-l})\bigr)_{I\in \mathcal{B}_R}$ defines an element of $\varprojlim R_g/(I\cap R_g)$. By completeness of $R_g$ we obtain $c_g\defeq\sum_{l\in G} a_l b_{g-l}\in R_g$, and $\gamma\defeq(c_g)\in \mathcal{H}$ since $G_I^\gamma=G_I^\alpha+G_I^\beta$ is finite. Thus $S$ is a ring and $\iota$ is well-defined.

Define $\varphi\colon S\to \widehat{R}\defeq\varprojlim_{I\in \mathcal{B}_R} R/I$ by
\[
\varphi(\alpha)=\bigl(\ \sum_{g\in G_I^\alpha}\pi_I(a_g)\ \bigr)_{I\in \mathcal{B}_R}.
\]
If $\varphi(\alpha)=0$, then $\pi_I(a_g)=0$ for all $I$, hence $a_g=0$ (completeness of $R_g$), so $\alpha=0$; thus $\varphi$ is injective. Conversely, given a compatible system $(a_I)_{I}$ with homogeneous decompositions $a_I=\sum_g a_{I,g}$, each $(a_{I,g})_I$ lifts to some $a_g\in R_g$, and $\alpha=(a_g)\in \mathcal{H}$ satisfies $\varphi(\alpha)=(a_I)_I$. Hence $\varphi$ is an isomorphism.
\end{proof}

Any cofiltered limits of graded topological ring can be written as follows.
Note that the topology on the limit is the limit topology, i.e., the limit of adic rings may be more complicated than the adic topology.

\begin{proposition}\label{gr-top-lim-const}
Let $\{R_j\}_{j\in J}$ be a cofiltered system of $G$-graded topological rings and set $\mathcal{B}_j\defeq\mathcal{B}_{R_j}$. Then:
\begin{enumerate}
    \item For each $g\in G$, let $R_g\defeq\varprojlim_{j\in J} (R_j)_g$ in abelian groups, and put $R\defeq\bigoplus_{g\in G} R_g$. Then $R$ is a $G$-graded ring.
    \item The projections $\rho_{j,g}\colon R_g\to (R_j)_g$ assemble to graded ring homomorphisms $\rho_j\colon R\to R_j$. Give $R$ the topology with fundamental system
    \[
      \mathcal{B}\defeq\{\ \rho_j^{-1}(I)\ \mid\ j\in J,\ I\in \mathcal{B}_j\ \}.
    \]
    \item With this topology, $R$ is the projective limit of $\{R_j\}_{j\in J}$ in the category of $G$-graded topological rings.
\end{enumerate}
\end{proposition}

\begin{proof}
(1)–(2) are immediate. For (3), let $S$ be a $G$-graded topological ring with compatible morphisms $\varphi_j\colon S\to R_j$. For each $g$, the maps $\varphi_{j,g}\colon S_g\to (R_j)_g$ induce a unique $\varphi_g\colon S_g\to R_g$ with $\rho_{j,g}\circ \varphi_g=\varphi_{j,g}$. These assemble to a graded ring homomorphism $\varphi\colon S\to R$. Continuity follows since for $\rho_j^{-1}(I)\in \mathcal{B}$ there exists a homogeneous neighborhood $I'\subset S$ with $\varphi(I')\subset \rho_j^{-1}(I)$. Uniqueness is clear from the universal property of limits in abelian groups.
\end{proof}

\begin{remark}
In the setting of \cref{gr-top-lim-const}, for each $g\in G$ the topological abelian group $R_g$ (with the induced topology) is the projective limit $\varprojlim_{j\in J} (R_j)_g$ in the category of topological abelian groups. Indeed, for all $j\in J$ and $I\in \mathcal{B}_j$ we have
\[
\rho_j^{-1}(I)\cap R_g=\rho_{j,g}^{-1}\bigl(I\cap (R_j)_g\bigr).
\]
\end{remark}

Using the limits of graded topological rings, we can formulate the \emph{tilt} and the \emph{ring of Witt vectors} of graded topological rings.
They will be used to construct the untilt and the corresponding graded perfect prisms of graded perfectoid rings as in \Cref{SharpMapThetaMap}.

\begin{example}\label{graded-witt-tilt}
Let $R$ be a $G$-graded topological ring.
\begin{enumerate}
    \item Assume $G=G[1/p]$. Define the $G$-graded topological ring $R^\flat_{\graded}$ as follows: for each $e\ge 0$,  via the canonical isomorphism 
    \[
    R/pR\simeq (R/pR)^{[p^{-e}]}=\bigoplus_{g \in G} R_{p^{-e}g}/pR_{p^{-e}g}
    \]
    of rings, endow $(R/pR)^{[p^{-e}]}$ with the induced topology; then $(R/pR)^{[p^{-e}]}$ is a $G$-graded topological ring. The Frobenius $F\colon (R/pR)^{[p^{-{e+1}}]}\to (R/pR)^{[p^{-e}]}$ is a morphism of $G$-graded topological rings. By \cref{gr-top-lim-const}, the projective limit exists in the category of $G$-graded topological rings; denote it by $R^{\flat,\graded}$. Explicitly,
    \[
      R^{\flat,\graded}
      = \bigoplus_{g\in G}\ \varprojlim_{e\ge 0}\{\cdots \xrightarrow{F} (R/pR)_{p^{-(e+1)}g}\xrightarrow{F}(R/pR)_{p^{-e}g}\xrightarrow{F}\cdots \xrightarrow{F}(R/pR)_g\},
    \]
    where the fundamental neighborhoods are inverse images along the projections $R^{\flat,\graded}\to (R/pR)^{[p^{-e}]}$. If $R$ is $p$-adically gradedwise complete, for each $e\ge 0$ and $g\in G$ the map
  \begin{align*}
        \sharp_{e, g} \colon \lim_{k \geq 0} \{\cdots \xrightarrow{F} (R/p)_{p^{-(k+1)}g} \xrightarrow{F} (R/p)_{p^{-k}g} \xrightarrow{F} \cdots \xrightarrow{F} (R/p)_g\} & \to R_{p^{-e}g} \\
        (\overline{a_k})_{k \geq 0} & \mapsto \lim_{k \to \infty} a_k^{p^{k-e}}
    \end{align*}
    is a well-defined multiplicative map, and taking direct sums in $g$ yields a graded multiplicative map $\sharp_{e,\graded}\colon R^{\flat,\graded}\to R^{[p^{-e}]}\defeq \bigoplus_{g\in G} R_{p^{-e}g}$. By construction, $\{\sharp_{e,\graded}((\overline{a_k})_{k\ge 0})\}_{e\ge 0}$ is a compatible system of $p$-power roots of $\sharp_{0,\graded}((\overline{a_k})_{k\ge 0})$ in $R$.
    \item Define the $G$-graded topological ring $W^{\graded}(R)$ as follows. Let $\mathcal{B}$ be a fundamental system of neighborhoods of $0$ in $R$ consisting of homogeneous ideals, closed under finite intersections. For $n\ge 1$, the Witt vectors $W_n(R)$ form a graded ring via
    \[
      W_n(R)_g=\{(a_0,\ldots,a_{n-1})\mid a_i\in R_{p^i g}\},
    \]
    and the restriction maps $W_{n+1}(R)\to W_n(R)$ are graded. Give $W_n(R)$ the topology with fundamental system $\mathcal{B}_n\defeq\{W_n(I)\}_{I\in \mathcal{B}}$; then each $W_n(R)$ is a $G$-graded topological ring and the restriction maps are morphisms. Define $W^{\graded}(R)$ to be the limit of $\{W_n(R)\}_{n\ge 1}$ in the category of $G$-graded topological rings (existence by \cref{gr-top-lim-const}). Then
    \[
      W^{\graded}(R)_g=\{(a_0,a_1,\ldots)\in W(R)\mid a_i\in R_{p^i g}\},
    \]
    and the topology on $W^{\graded}(R)$ is generated by $\{W(I)\}_{I\in \mathcal{B}}$ and $\{V^k W^{\graded}(R)\}_{k\ge 1}$ where $V^k W^{\graded}(R)\defeq\ker(W^{\graded}(R)\to W_k(R))$. For each $n\ge 1$,
    \[
      W_n(R)^{\wedge} \cong \varprojlim_{I\in \mathcal{B}} W_n(R/I) \cong W_n\!\bigl(\varprojlim_{I\in \mathcal{B}} R/I\bigr)=W_n(R^{\wedge}),
    \]
    and the gradedwise completion satisfies
    \[
      W_n(R)^{\grwedge}
      \cong \bigoplus_{g\in G}\ \varprojlim_{I\in \mathcal{B}} W_n(R/I)_g
      \cong W_n\!\bigl(R^{\grwedge}\bigr).
    \]
\end{enumerate}
\end{example}

\subsection{Pro-graded rings}

In this subsection, we will introduce the notion of \emph{pro-\(G\)-graded rings} so that we can treat the completion of graded topological rings more fluently to reach our main theorem.

\begin{definition}\label{def-pro-gra}
We define the notion of pro-\(G\)-graded rings:
\begin{itemize}
    \item A \emph{pro-\(G\)-graded ring} is a pair \((R, R_{\graded})\) of topological rings satisfying the following conditions:
    \begin{enumerate}
        \item \(R_{\graded}\) is a gradedwise complete \(G\)-graded topological ring.
        \item The completion of \(R_{\graded}\) with respect to its topology coincides with \(R\).
    \end{enumerate}
    We often write \(R\) for \((R, R_{\graded})\) when there is no confusion, and call \(R_{\graded}\) the \emph{structure graded ring} of \(R\).
    The degree-\(g\) part of \(R_{\graded}\) is denoted by \(R_g\).

    \item If a pro-\(G\)-graded ring \((R, R_{\graded})\) is such that \(R_{\graded}\) is a \(G\)-graded adic ring in the sense of \Cref{DefTopGraded}, then we say that \((R, R_{\graded})\) is a \emph{pro-\(G\)-graded adic ring}.

    \item A \emph{pro-\(G\)-graded homomorphism} \((R, R_{\graded}) \to (S, S_{\graded})\) is a pair \((\varphi, \varphi_{\graded})\) consisting of a ring homomorphism \(\varphi \colon R \to S\) and a continuous \(G\)-graded ring homomorphism \(\varphi_{\graded} \colon R_{\graded} \to S_{\graded}\) such that the induced ring homomorphism \(\widehat{\varphi_{\graded}} \colon R \to S\) coincides with \(\varphi\). 

    \item Let \((R, R_{\graded})\) be a pro-\(G\)-graded ring and let \(I\) be an ideal of \(R\).
    We say that \(I\) is a \emph{homogeneous ideal} if \(I = J R\) for some homogeneous ideal \(J \subseteq R_{\graded}\).
\end{itemize}
\end{definition}

\begin{remark}\label{remk-first-pro-gra}
Let \((R, R_{\graded})\) be a pro-\(G\)-graded ring, and set \(\mathcal{B} \defeq \mathcal{B}_R\).
\begin{itemize}
    \item The completion of \(R_{\graded}\) is given by \(\varprojlim_{I \in \mathcal{B}} R_{\graded} / I\), so condition~\textup{(2)} in \cref{def-pro-gra} means that \(R \simeq \varprojlim_{I \in \mathcal{B}} R_{\graded} / I\).
    
    \item The natural homomorphism 
    \[
    \iota \colon R_{\graded} \to R, \qquad 
    \text{given by } R_{\graded} \to \widehat{R_{\graded}} \xrightarrow{\sim} R,
    \]
    is injective.  
    Indeed, suppose \(a \in \ker(\iota)\). Then \(a \in I\) for all \(I \in \mathcal{B}\).  
    Writing \(a = \sum_{g \in G} a_g\) as its homogeneous decomposition in \(R_{\graded}\), and noting that each \(I\) is homogeneous, we have \(a_g \in I \cap R_g\) for all \(I \in \mathcal{B}\) and \(g \in G\).  
    Since each \(R_g\) is complete with respect to the induced topology from \(R_{\graded}\), it follows that \(a_g = 0\) for all \(g\), hence \(a = 0\).

    \item The family \(\{\overline{IR} \mid I \in \mathcal{B}\}\) forms a fundamental system of neighborhoods of \(0\) for the topology of \(R\) induced by the isomorphism \(\varprojlim_{I \in \mathcal{B}} R / I \simeq R\), 
    where \(\overline{IR}\) denotes the topological closure of \(IR\) in \(R\).  
    Indeed, it suffices to show that \(\overline{IR} = \ker(R \to R_{\graded} / I)\) for \(I \in \mathcal{B}\).
    The inclusion \(\overline{IR} \subseteq \ker(R \to R_{\graded} / I)\) follows from the completeness of \(R_{\graded} / I\).
    Conversely, for \(a \in \ker(R \to R_{\graded} / I)\), we can write \(a = (a_J)_{J \in \mathcal{B}}\) with \(a_J \in R_{\graded} / J\).
    Since \(a \in \ker(R \to R_{\graded} / I)\), we have \(a_J \in I\) for all \(J \in \mathcal{B}\), and hence \(a \in \overline{IR}\), as desired.
    
    \item We obtain the categorical equivalent between the category of gradedwise complete $G$-graded topological ring and the category of pro-$G$-graded rings given by 
    \begin{equation*}
        S \to (S^{\wedge},S) \qquad \varphi \mapsto (\widehat{\varphi},\varphi) \quad \text{and} \quad (S,S_{\gr}) \mapsto S_{\gr} \qquad (\varphi,\varphi_{\gr}) \mapsto \varphi_{\gr},
    \end{equation*}
    where \(S^{\wedge}\) is the completion of a \(G\)-graded topological ring \(S\).
    Essentially, these two notions are the same but thinking about graded topological rings with their completion is more convenient to prove, at least, our main theorems.
\end{itemize}
\end{remark}

\begin{example}
Let \(R\) be a \(G\)-graded ring.  
Then \((R, R)\) is a pro-\(G\)-graded ring equipped with the discrete topology.  
Any pro-\(G\)-graded homomorphism from \((R, R)\) to another pro-\(G\)-graded ring \((S, S_{\graded})\) corresponds uniquely to a \(G\)-graded ring homomorphism \(R \to S_{\graded}\).
\end{example}

\begin{remark} \label{RemarkGRGradedRing}
    In \cite[Definition 8.5.1]{gabber2018Foundations}, they introduced the notion of \emph{\(G\)-graded structure} on a separated topological ring \(R\).
    The category of $G$-graded structure on a complete topological rings is equivalent to the one of pro-$G$-graded rings. 
    However, we developed our definition independently before becoming aware of their work.
    We retain our terminology and provide complete proofs for consistency throughout this article.
\end{remark}

\begin{lemma}\label{LocalizationProGraded}
Let \((R, R_{\graded})\) be a pro-\(G\)-graded ring, and let \(W\) be a multiplicative subset of homogeneous elements of \(R_{\graded}\).
Then the canonical morphism 
\[
R_{\graded} \longrightarrow W^{-1}R_{\graded}
\]
is a graded morphism of \(G\)-graded rings and endows \(W^{-1}R_{\graded}\) with a natural linear topology.  
The pair
\[
\widehat{W^{-1}R} \defeq \bigl(\widehat{W^{-1}R_{\graded}}, (W^{-1}R_{\graded})^{\grwedge}\bigr)
\]
is again a pro-\(G\)-graded ring.  
If \(R\) is adic, then so is \(\widehat{W^{-1}R_{\graded}}\).  
Moreover, \(\widehat{W^{-1}R}\) is the universal pro-\(G\)-graded \(R\)-algebra in which the elements of \(W\) become invertible.
\end{lemma}

\begin{proof}
This follows directly from the definitions.
\end{proof}

As mentioned before, elements of the completion of graded topological rings have ``homogeneous decompositions'' in the following sense:

\begin{definition}\label{defn-homog-decomp}
Let \((R, R_{\graded})\) be a pro-\(G\)-graded ring,  let \(a \in R\) and $\cB\defeq\cB_R$.

By \cref{const-comp}, there uniquely exists a family \((a_g)_{g \in G} \in \prod_{g \in G} R_g\) such that:
\begin{itemize}
  \item for every \(I \in \mathcal{B}\), the set
  \[
  G_I \defeq \{\, g \in G \mid a_g \notin I \,\}
  \]
  is finite; and
  \item the element \(a\) corresponds to the projective system
  \[
  \Bigl( \sum_{g \in G_I} \pi_I(a_g) \Bigr)_{I \in \mathcal{B}}
  \qquad\text{under }\qquad
  R \simeq \varprojlim_{I \in \mathcal{B}} R / I.
  \]
\end{itemize}

In this situation, we write
\[
a = \sum_{g \in G} a_g
\]
and call this expression the \emph{topological homogeneous decomposition} of \(a\).
Each \(a_g\) is called the \emph{degree-\(g\) part} of \(a\).
\end{definition}

\begin{remark}\label{rmk-hom-decom}
In the setting of \cref{defn-homog-decomp}, let \(a, b \in R\) with topological homogeneous decompositions
\[
a = \sum_{g \in G} a_g,
\qquad
b = \sum_{g \in G} b_g.
\]
Then their sum and product decompose as
\[
a + b = \sum_{g \in G} (a_g + b_g),
\qquad
ab = \sum_{g \in G} \Bigl( \sum_{l \in G} a_l b_{g - l} \Bigr),
\]
as follows from \cref{const-comp}.
\end{remark}

Using topological homogeneous decompositions is an effective way to consider the complete part of pro-\(G\)-graded rings.
The next propositions are one of the examples of such phenomena.

\begin{proposition}\label{purity-homog-ideal}
Let \((R, R_{\graded})\) be a pro-\(G\)-graded ring and let \(J \subseteq R_{\graded}\) be a homogeneous ideal.
\begin{enumerate}
    \item For any \(a \in JR\), the degree-\(g\) part \(a_g\) of \(a\) lies in \(J\) for every \(g \in G\).  
    In particular, \( JR \cap R_{\graded} = J \).

    \item We have \( \overline{JR} \cap R_{\graded} = \overline{J} \),  
    where \(\overline{JR}\) denotes the topological closure of \(JR\) in \(R\)  
    and \(\overline{J}\) the closure of \(J\) in \(R_{\graded}\).
\end{enumerate}
\end{proposition}

\begin{proof}
(1)  
Take \(a \in JR\). Then we can write
\[
a = \sum_m x_m b_m
\]
for some \(x_m \in J\) and \(b_m \in R\).
Let 
\[
a = \sum_{g \in G} a_g, \quad
x_m = \sum_{g \in G} x_{m,g}, \quad
b_m = \sum_{g \in G} b_{m,g}
\]
be their topological homogeneous decompositions.  
By \cref{rmk-hom-decom}, comparing the degree-\(g\) parts in the equality \(a = \sum_m x_m b_m\), we obtain
\begin{equation}\label{eq:comp-hom-decom}
  a_g = \sum_m \sum_{l \in G} x_{m,l} b_{m,g-l}.
\end{equation}
Since each \(x_m \in R_{\graded}\), only finitely many \(x_{m,l}\) are nonzero, so the right-hand side of \eqref{eq:comp-hom-decom} is a finite sum.  
As \(J\) is a homogeneous ideal and \(x_{m,l} \in J\), \(b_{m,g-l} \in R_{\graded}\), it follows that \(a_g \in J\).

To prove the final statement of (1), note that \(J \subseteq JR \cap R_{\graded}\) is clear.  
Conversely, let \(a \in JR \cap R_{\graded}\).  
By the first part, in its homogeneous decomposition \(a = \sum a_g\) each \(a_g \in J\).  
Since \(a \in R_{\graded}\), the sum is finite, hence \(a \in J\), as required.

\smallskip
(2)  
Let \(\mathcal{B}\) be a fundamental system of neighborhoods of \(0\) in \(R\) consisting of homogeneous ideals, closed under finite intersections.  
Then for each \(I \in \mathcal{B}\), we have natural isomorphisms
\[
R_{\graded}/(I + J)
  \simeq R_{\graded}/I \otimes_{R_{\graded}} R_{\graded}/J
  \simeq R/\overline{IR} \otimes_R R/JR
  \simeq R/(\overline{IR} + JR),
\]
by \cref{remk-first-pro-gra}.  
This implies that
\[
\widehat{R_{\graded}/J} \simeq \widehat{R/JR}.
\]
Therefore,
\[
\overline{JR} \cap R_{\graded}
  = \Ker(R \to \widehat{R/JR}) \cap R_{\graded}
  = \Ker(R_{\graded} \to \widehat{R_{\graded}/J})
  = \overline{J},
\]
as claimed.
\end{proof}

\begin{proposition}\label{quot-pro-gra}
Let \((R, R_{\graded})\) be a pro-\(G\)-graded ring.
\begin{enumerate}
    \item Let \((\varphi, \varphi_{\graded}) \colon (R, R_{\graded}) \to (S, S_{\graded})\) be a pro-\(G\)-graded homomorphism to another pro-\(G\)-graded ring \((S, S_{\graded})\).  
    Then \(\Ker(\varphi)\) is the topological closure of \(\Ker(\varphi_{\graded}) R_{\graded}\) in \(R\).

    \item Let \(J\) be a homogeneous ideal of \(R\), and let \(\overline{J}\) denote its topological closure.  
    Then \((R / \overline{J},\, R_{\graded} / \overline{J \cap R_{\graded}})\) is a pro-\(G\)-graded ring.
\end{enumerate}
\end{proposition}

\begin{proof} 
(1)  
Take \(a \in \Ker(\varphi)\) and let \(a = \sum a_g\) be its topological homogeneous decomposition.  
Then
\[
\varphi(a) = \sum \varphi_{\graded}(a_g)
\]
is the homogeneous decomposition of \(\varphi(a) = 0\).  
Hence \(a_g \in \Ker(\varphi_{\graded})\) for all \(g \in G\),  
and therefore \(a\) lies in the topological closure of \(\Ker(\varphi_{\graded}) R_{\graded}\), as claimed.

\smallskip
(2)  
Let \(J_{\graded} \subseteq R_{\graded}\) be a homogeneous ideal such that \(J = J_{\graded} R\).  
By \cref{purity-homog-ideal}(2), we have an injection
\[
R_{\graded} / \overline{J_{\graded}} \hookrightarrow R / \overline{J}.
\]
Since \(J_{\graded}\) is homogeneous, its closure is also homogeneous:
\[
\overline{J_{\graded}}
  = \overline{\bigoplus (R_g \cap J_{\graded})}
  = \bigoplus \overline{R_g \cap J_{\graded}}.
\]
Thus, \(R_{\graded} / \overline{J_{\graded}}\) is a \(G\)-graded ring.

It remains to show that \(R / \overline{J}\) is the completion of \(R_{\graded} / \overline{J_{\graded}}\) with respect to the quotient topology.  
Let \(\mathcal{B}\) be a fundamental system of neighborhoods of \(0\) in \(R\) consisting of homogeneous ideals, and assume \(\mathcal{B}\) is closed under finite intersections.  
Then for each \(I \in \mathcal{B}\), we have
\[
R_{\graded} / (I + \overline{J_{\graded}})
  \simeq R / (\overline{IR} + \overline{J_{\graded}} R),
\]
and hence
\[
\varprojlim_{I \in \mathcal{B}} R_{\graded} / (I + \overline{J_{\graded}})
  \simeq
  \varprojlim_{I \in \mathcal{B}} R / (\overline{IR} + \overline{J_{\graded}} R).
\]
This is the completion of \(R_{\graded} / \overline{J_{\graded}}\).  

Now consider the short exact sequence
\[
0 \to
  \overline{J_{\graded}} R / (\overline{IR} \cap \overline{J_{\graded}} R)
  \to
  R / \overline{IR}
  \to
  R / (\overline{IR} + \overline{J_{\graded}} R)
  \to 0,
\]
for every \(I \in \mathcal{B}\).  
Passing to the projective limit, we obtain an exact sequence
\[
0 \to \overline{J}
  \to R
  \to \varprojlim_{I \in \mathcal{B}} R / (\overline{IR} + \overline{J_{\graded}} R)
  \to 0.
\]
This shows that \(R / \overline{J}\) is the completion of \(R_{\graded} / \overline{J_{\graded}}\), as claimed.
\end{proof}

Although the topology on pro-\(G\)-graded rings are not necessarily adic, we have a variant of topological nakayama lemma.
Compare the adic ones in \Cref{graded-top-nakayama}.

\begin{lemma} \label{SurjectiveProGraded}
    Let \((\varphi, \varphi_{\graded}) \colon (R, R_{\graded}) \to (S, S_{\graded})\) be a morphism of pro-\(G\)-graded rings.
    If \(\varphi \colon R \to S\) is surjective, then so is \(\varphi_{\graded} \colon R_{\graded} \to S_{\graded}\)
\end{lemma}

\begin{proof}
    Take any homogeneous element \(x \in S_{\graded}\).
    Because of the surjectivity of \(\varphi\), we can take an element \(y\) of \(R\) such that \(\varphi(y) = x\).
    Take the topological homogeneous decomposition \(y = \sum_{g \in G} y_g\) in \(R\).
    Since \(\varphi(y)\) is the same as \(\sum_{g \in G} \varphi_{\graded}(y_g)\), the equality \(\varphi(y) = x\) in \(S\) implies \(\varphi_{\graded}(y_g) = 0\) for any \(g \neq \deg(x)\) and \(\varphi_{\graded}(y_{\deg(x)}) = x\).
    Then the homogeneous element \(y_{\deg(x)} \in R_{\graded}\) goes to \(x\) via \(\varphi_{\graded}\) and this shows the surjectivity of \(\varphi_{\graded}\).
\end{proof}

In the rest of this subsection, we will compute cofiltered limits of pro-\(G\)-graded rings and take the tilt, the ring of Witt vectors, and the Fontaine's theta map in the pro-\(G\)-graded setting.

\begin{proposition}\label{lim-pro-is-pro}
The category of pro-\(G\)-graded rings has cofiltered limits. 
More precisely, let \(\{(R_j,R_{j,\graded})\}_{j \in J}\) be a cofiltered system of pro-\(G\)-graded rings.
Set \(R \defeq \varprojlim_{j \in J} R_j\) (the limit in the category of rings) and 
\(R_{\graded} \defeq \bigoplus_{g \in G} \varprojlim_{j \in J} (R_{j,\graded})_g\).
Then the canonical map \( \iota \colon R_{\graded} \to R \) is injective and 
\((R,R_{\graded})\) is a pro-\(G\)-graded ring with the following properties:
\begin{enumerate}
    \item The linear (homogeneous) topology on \(R_{\graded}\) is defined by the family of kernels
    \[
      \Ker(R_{\graded} \xrightarrow{\pi_{j}} R_{j,\graded} \twoheadrightarrow R_{j,\graded}/I_{j,i}),
    \]
    where \(\{I_{j,i}\}_{i \in \Sigma_j}\) is a fundamental system of open neighborhoods of \(0\) in \(R_{j,\graded}\) for each \(j \in J\).
    \item For every \(j \in J\), the projection \(R \to R_j\) induces a pro-\(G\)-graded homomorphism
    \((R,R_{\graded}) \to (R_j,R_{j,\graded})\).
    \item These morphisms exhibit \((R,R_{\graded})\) as the cofiltered limit \(\varprojlim_{j \in J} (R_j,R_{j,\graded})\) in the category of pro-\(G\)-graded rings.
\end{enumerate}
\end{proposition}

\begin{remark} \label{CatEquivProG}
Before the proof, we explain the relation with \Cref{gr-top-lim-const}.
Since the completion of a given topological ring is unique, the completion functor yields an equivalence between the category of gradedwise complete \(G\)-graded topological rings and the category of pro-\(G\)-graded rings.

Consequently, \Cref{gr-top-lim-const} implies that pro-\(G\)-graded rings admit cofiltered limits, and that the limit 
\(\varprojlim_{j \in J} (R_j, R_{j,\graded})\) is obtained as follows: the structure graded ring is
\(R_{\graded} \defeq \bigoplus_{g \in G} \varprojlim_{j \in J} (R_{j,\graded})_g\) endowed with the limit topology along the maps \(R_{\graded} \to R_{j,\graded}\), and \(R\) is its completion.
A priori it is not clear that this completion \(R\) coincides with the projective limit \(\varprojlim_{j \in J} R_j\) of complete topological rings; \Cref{lim-pro-is-pro} asserts that this is in fact true.

Categorically, \Cref{gr-top-lim-const} and \Cref{lim-pro-is-pro} together show that the forgetful functor from pro-\(G\)-graded rings to (topological) rings commutes with cofiltered limits.
\end{remark}

\begin{proof}[Proof of \Cref{lim-pro-is-pro}]
For each \(g \in G\), the system \(\{(R_{j,\graded})_g\}_{j \in J}\) is cofiltered in abelian groups; set
\(R_g \defeq \varprojlim_{j \in J} (R_{j,\graded})_g\) and \(R_{\graded} \defeq \bigoplus_{g \in G} R_g\).

\smallskip
\textbf{Injectivity of \(\iota\colon R_{\graded}\hookrightarrow R\).}
For each \(j\) and \(g\), the inclusion \((R_{j,\graded})_g \hookrightarrow R_{j,\graded}\) induces \(R_g \hookrightarrow R\).
Summing over \(g\), we obtain \(\iota \colon R_{\graded} \to R\) compatible with the projections:
\[
\begin{tikzcd}
R_g \arrow[r, hook] \arrow[d] & R_{\graded} \arrow[r, "\iota"] \arrow[d, "\pi_j"] & R \arrow[d] \\
(R_{j,\graded})_g \arrow[r, hook] & R_{j,\graded} \arrow[r, hook] & R_j
\end{tikzcd}
\]
Write \(a=\sum_{g} a_g \in R_{\graded}\).
If \(\iota(a)=0\), then for each \(j\) we have \(0=\sum_{g} \pi_j(a_g)\) in \(R_j\).
Since this is a homogeneous decomposition and \(R_{j,\graded}\hookrightarrow R_j\) is injective, it follows that \(\pi_j(a_g)=0\) for all \(j,g\).
Thus \(a_g=0\) in \(R_g=\varprojlim_j (R_{j,\graded})_g\) for each \(g\), hence \(a=0\). Therefore \(\iota\) is injective.


\smallskip
\textbf{Surjectivity of $\iota$:}
For each \(j\), fix a cofiltered fundamental system \(\{I_{j,i}\}_{i\in\Sigma_j}\) of open homogeneous ideals in \(R_{j,\graded}\) with \(\varprojlim_i R_{j,\graded}/I_{j,i} \cong R_j\).
Let \(\Sigma\defeq\{(j,i) \mid j\in J,\ i\in\Sigma_j\}\) and define
\[
K_{j,i} \defeq \Ker(R_{\graded} \xrightarrow{\pi_j} R_{j,\graded} \twoheadrightarrow R_{j,\graded}/I_{j,i})
\quad\text{for }(j,i)\in\Sigma.
\]
Endow \(R_{\graded}\) with the linear topology generated by \(\{K_{j,i}\}_{(j,i)\in\Sigma}\). 
Then, for each \(g\),
\[
R_g \cong \varprojlim_{(j,i)\in\Sigma} R_g/(K_{j,i}\cap R_g),
\]
so each \(R_g\) is complete for the induced topology and \(R_{\graded}\) is gradedwise complete.

Finally, using the identifications
\[
R = \varprojlim_{j\in J} R_j \cong \varprojlim_{(j,i)\in\Sigma} R_{j,\graded}/I_{j,i},
\qquad
R_{\graded} = \bigoplus_{g\in G} \varprojlim_{j\in J} (R_{j,\graded})_g,
\]
one checks that the completion of \(R_{\graded}\) with respect to \(\{K_{j,i}\}\) is canonically isomorphic to \(R\):
every \(x\in R\) admits a (topological) homogeneous decomposition \(x_j=\sum_{g} x_{j,g}\) in each \(R_j\), compatible in \(j\), hence defining a coherent family $((x_{j,g})_j)_g$ in \(\prod_g R_g\) which maps to \(x\).
Therefore, we obtain the surjectivity of $\iota$, thus \((R,R_{\graded})\) is a pro-\(G\)-graded ring.

The assertions \textup{(2)} and \textup{(3)} follow from the construction and the universal property of projective limits. 
\end{proof}

\begin{lemma}\label{pro-graded-tilt}
Let \((R, R_{\graded})\) be a pro-\(G\)-graded ring such that \(R_{\graded}\) is a \(G\)-graded adic ring with a homogeneous ideal of definition \(J\).
Assume \(G = G[1/p]\) and that \(p\) is topologically nilpotent in \(R_{\graded}\).
Then the pair \((R^\flat, R^{\flat,\graded}_{\graded})\) is a pro-\(G\)-graded ring, where \(R^{\flat,\graded}_{\graded}\) is defined in \Cref{graded-witt-tilt}(1) and \(R^\flat = \varprojlim_{e \ge 0} R/pR\) is the tilt of \(R\) equipped with the limit topology along the projections \(R^\flat \to R/pR\).
Moreover, the morphism \(\sharp_{e,\graded} \colon R^{\flat,\graded}_{\graded} \to R_{\graded}\) is compatible with the map \(\sharp^{1/p^e} \colon R^\flat \to R\) for each \(e \ge 0\).
\end{lemma}

\begin{proof}
Recall from \Cref{graded-witt-tilt}(1) that the \(G\)-graded topological ring \(R^{\flat,\graded}_{\graded}\) is the projective limit of the system \(\{(R_{\graded}/pR_{\graded})^{[p^{-e}]}\}_{e \ge 0}\), where each
\[
(R_{\graded}/pR_{\graded})^{[p^{-e}]} = \bigoplus_{g \in G} (R_{p^{-e}g} / pR_{p^{-e}g})
\]
is canonically isomorphic to \(R_{\graded} / pR_{\graded}\) as a ring.
Hence, we obtain a projective system \(\{(R/pR, (R_{\graded}/pR_{\graded})^{[p^{-e}]})\}_{e \ge 0}\) of pairs consisting of the topological ring \(R/pR\) and the \(G\)-graded adic ring \((R_{\graded}/pR_{\graded})^{[p^{-e}]}\), whose transition maps are Frobenius morphisms.

\smallskip
We compute the gradedwise completion \((R_{\graded}/pR_{\graded})^{\grwedge}\) with respect to the homogeneous ideal \(J\) as in \Cref{GradedwiseCompletion}.
Since \(R_{\graded}\) is already gradedwise complete, the completion of each graded component \(R_g/pR_g\) is \(R_g / \overline{pR_g}\), where \(\overline{pR_g}\) denotes the topological closure.
Thus the gradedwise completion satisfies
\[
(R_{\graded}/pR_{\graded})^{\grwedge} \cong R_{\graded} / \overline{pR_{\graded}},
\]
whose degree-\(g\) part is \(R_g / \overline{pR_g}\).
The same computation applies to each \((R_{\graded}/pR_{\graded})^{[p^{-e}]}\), giving
\(((R_{\graded}/pR_{\graded})^{[p^{-e}]})^{\grwedge} \cong (R_{\graded}/\overline{pR_{\graded}})^{[p^{-e}]}\) for all \(e \ge 0\).

\smallskip
Since \(R\) is the \(J\)-adic completion of \(R_{\graded}\), the \(J\)-adic completion of \(R_{\graded}/\overline{pR_{\graded}}\) agrees with that of \(R/pR\), which is \(R/\overline{pR}\).
Hence we have a projective system
\[
\{(R/\overline{pR}, (R_{\graded}/\overline{pR_{\graded}})^{[p^{-e}]})\}_{e \ge 0}
\]
of pro-\(G\)-graded rings whose transition maps are Frobenius morphisms.

\smallskip
Applying \Cref{lim-pro-is-pro}, we obtain
\[
\varprojlim_{e \ge 0} (R/\overline{pR}, (R_{\graded}/\overline{pR_{\graded}})^{[p^{-e}]})
  \simeq
  \bigl(\varprojlim_{e \ge 0} R/\overline{pR},\, \varprojlim_{e \ge 0} (R_{\graded}/\overline{pR_{\graded}})^{[p^{-e}]}\bigr),
\]
which is a pro-\(G\)-graded ring.
It remains to identify this with \((R^\flat, R^{\flat,\graded}_{\graded})\).

\smallskip
We have a short exact sequence of topological rings
\[
0 \longrightarrow \overline{pR}/pR
  \longrightarrow R/pR
  \longrightarrow R/\overline{pR}
  \longrightarrow 0.
\]
Since \(R\) is \(J\)-adically complete, \(R/pR\) is derived \(J\)-complete.  
Hence the first term is killed by a power of \(p\) by \cite{stacks-project}*{Tag~0G3I}, and taking the projective limit along the Frobenius morphisms yields
\[
\varprojlim_{e \ge 0} R/\overline{pR} \cong R^\flat
\]
as topological rings.

The same argument applies degreewise to each graded piece.
From the exact sequence
\[
0 \longrightarrow \overline{pR_{p^{-e}g}} / pR_{p^{-e}g}
  \longrightarrow R_{p^{-e}g} / pR_{p^{-e}g}
  \longrightarrow R_{p^{-e}g} / \overline{pR_{p^{-e}g}}
  \longrightarrow 0,
\]
we deduce that
\[
\varprojlim_{e \ge 0} (R_{\graded} / \overline{pR_{\graded}})^{[p^{-e}]} \cong R^{\flat,\graded}_{\graded}
\]
as \(G\)-graded topological rings.

Finally, the maps \(\sharp_{e,\graded} \colon R^{\flat,\graded}_{\graded} \to R_{\graded}\) are compatible with \(\sharp^{1/p^e} \colon R^\flat \to R\) under these identifications, completing the proof.
\end{proof}

\begin{lemma}\label{pro-graded-Witt}
Let \((R, R_{\graded})\) be a pro-\(G\)-graded ring such that \(R_{\graded}\) is a \(G\)-graded adic ring with a homogeneous ideal of definition \(J\).
Then the pair \((W(R), W^{\graded}(R_{\graded}))\) is a pro-\(G\)-graded ring, where \(W(R)\) is the Witt vector ring of \(R\) and \(W^{\graded}(R_{\graded})\) is the \(G\)-graded ring defined in \Cref{graded-witt-tilt}(2), endowed with the topology generated by \(\{W(J^k)\}_{k \ge 1}\) and \(\{V^k W^{\graded}(R_{\graded})\}_{k \ge 1}\).
\end{lemma}

\begin{proof}
For each \(n \ge 1\), the pair \((W_n(R), W_n(R_{\graded}))\) forms a pro-\(G\)-graded ring by \Cref{graded-witt-tilt}(2).
By \Cref{lim-pro-is-pro}, the limit of the projective system \(\{(W_n(R), W_n(R_{\graded}))\}_{n \ge 1}\) in the category of pro-\(G\)-graded rings is \((W(R), W^{\graded}(R_{\graded}))\).
\end{proof}

\begin{construction}\label{SharpMapThetaMap}
Let \((R, R_{\graded})\) be a pro-\(G\)-graded ring such that \(R_{\graded}\) is a \(G\)-graded adic ring with a homogeneous ideal of definition \(J\).
Assume \(G = G[1/p]\) and that \(p\) is topologically nilpotent in \(R_{\graded}\).


Since the maps \(\sharp_{e,\graded} \colon R^{\flat,\graded}_{\graded} \to R_{\graded}^{[p^{-e}]}\) are multiplicative \(G\)-graded morphisms by \Cref{graded-witt-tilt}(1), the usual arguments on non-graded objects shows that we can obtain a \(G\)-graded ring homomorphism
\begin{align*}
  \theta_{\graded} \colon A_{\inf}^{\graded}(R_{\graded}) \defeq W(R^{\flat,\graded}_{\graded})
  &\longrightarrow R_{\graded}, \\
  (a_0, a_1, \dots) &\longmapsto 
    \sharp_{0,\graded}(a_0) + p \sharp_{1,\graded}(a_1)
    + p^2 \sharp_{2,\graded}(a_2) + \cdots.
\end{align*}

Moreover, the compatibility between 
\(\sharp_{0,\graded} \colon R^{\flat,\graded}_{\graded} \to R_{\graded}\)
and
\(\sharp \colon R^\flat \to R\)
implies that the following diagram commutes:
\begin{center}
  \begin{tikzcd}
    W(R^\flat) \arrow[r, "\theta"] 
      & R \\
    W^{\graded}(R^{\flat,\graded}_{\graded}) 
      \arrow[r, "\theta_{\graded}"] \arrow[u, hook] 
      & R_{\graded} \arrow[u, hook]
  \end{tikzcd}
\end{center}
where \(\theta \colon W(R^\flat) \to R\) is the Fontaine theta map.
Hence we obtain a morphism of pro-\(G\)-graded rings
\[
(\theta, \theta_{\graded}) \colon (A_{\inf}(R), A_{\inf}^{\graded}(R_{\graded})) \longrightarrow (R, R_{\graded}).
\]
\end{construction}

\section{Graded perfectoid rings}

This section is devoted to introduce the notion of graded perfectoid rings and the proof of a graded variant of Andr\'e's flatness lemma together with the existence of the ``graded perfectoidization'' of semiperfectoid rings.

\subsection{Graded perfectoid rings}

\begin{definition} \label{DefGradedPerfectoid}
We define a graded variant of perfectoid rings.
\begin{itemize}
\item Let $R$ be a $G$-graded ring.
We say that $R$ is \emph{a graded perfectoid ring} if $R$ is $p$-adically gradedwise complete and $R^{\wedge p}$ is perfectoid.
    
\item Let $R=(R,R_{\graded})$ be a pro-$G$-graded adic ring.
We say that $R$ is \emph{a pro-$G$-graded perfectoid ring} if $p$ is topological nilpotent in $R_{\graded}$ and $R_{\graded}$ is a $G$-graded perfectoid ring.
\end{itemize}
\end{definition}

\begin{remark}\label{first-cat-equiv}
The categorical equivalence in \cref{remk-first-pro-gra} induces the following categorical equivalence
\[
(\text{category of $G$-graded perfectoid}) \xrightarrow{\simeq} (\text{category of pro-$G$-graded $p$-adic perfectoid})
\]
given by $R \mapsto (R^{\wedge p},R)$, and the quasi-inverse is given by $(R,R_{\gr}) \to R_{\gr}$.
\end{remark}

\begin{lemma}\label{p-th-power}
Let \((R, R_{\graded})\) be a pro-\(G\)-graded ring, and let \(a \in R\).
Take a topological homogeneous decomposition \(a = \sum_{g \in G} a_g\).
\begin{enumerate}
    \item We have \(a^p \equiv \sum_{g \in G} a_g^p \pmod{pR}\).
    \item If $a$ and $p$ are topologically nilpotent, then each $a_g$ for $g \in G$ is topologically nilpotent, and so is 
\[
\frac{a^p - \sum_{g \in G} a_g^p}{p}
\]
in $R$.
    \item If \(a\) is \(p\)-torsion, then \(a^p = \sum_{g \in G} a_g^p\).
\end{enumerate}
In particular, we obtain two graded homomorphisms
\[
R_{\graded}/p R_{\graded}
        \xrightarrow{\,a \mapsto a^p\,}
        (R_{\graded}/p R_{\graded})^{[p]}, 
        \qquad R_{\graded}[p]
        \xrightarrow{\,a \mapsto a^p\,}
        (R_{\graded}[p])^{[p]}
\]
of rings and multiplicative monoids respectively.
\end{lemma}

\begin{proof}
We first prove (1).  
For each \(I \in \mathcal{B}_R\), let \(G_I^a \subseteq G\) be as in \cref{const-comp}.
Set
\[
a_I \defeq \sum_{g \in G_I^a} a_g,
\]
so that \(\lim_I a_I = a\).
Define
\[
a_I^{(p)} \defeq \sum_{g \in G_I^a} a_g^p,
\]
then \(\lim_I a_I^{(p)} = \sum_{g \in G} a_g^p\), since \(a_h \in I\) implies \(a_h^p \in I\).

Write \(G_I^a = \{g_1, \ldots, g_{r_I}\}\), and set
\[
b_I \defeq 
\sum_{\substack{0 \le \alpha_1, \ldots, \alpha_{r_I} \le p-1 \\ \alpha_1 + \cdots + \alpha_{r_I} = p}}
\frac{1}{p} \binom{p}{\alpha_1, \ldots, \alpha_{r_I}}
a_{g_{1}}^{\alpha_1} \cdots a_{g_{r_I}}^{\alpha_{r_I}}.
\]
Then we have \(a_I^p = a_I^{(p)} + p b_I\).
Moreover, by the definition of \(b_I\), for \(I, J \in \mathcal{B}_R\) with \(J \subseteq I\), we have
\[
b_J - b_I \in (\, a_h \mid h \in G_J^a \setminus G_I^a \,)
= (\, a_h \mid a_h \in I \setminus J \,) \subseteq I.
\]
Thus, \(\{b_I\}\) is a Cauchy net.
Therefore,
\[
a^p = \sum_{g \in G} a_g^p + p \lim_I b_I
  \equiv \sum_{g \in G} a_g^p \pmod{pR},
\]
as desired.

Next, we assume that $a$ and $p$ are topologically nilpotent to prove (2).
We first prove that each $a_g$ is topologically nilpotent for every $g \in G$.
Let $J \in \mathcal{B}_R$.
Since $a$ is topologically nilpotent, there exists an integer $n$ such that $a^{p^n} \in J$.
For every $g \in G$, we have 
\[
(a^{p^n})_{p^n g} \equiv a_g^{p^n} \pmod{pR}
\]
by (1). 
Since $J$ is homogeneous, we have $a_g^{p^n} \in (J,p)R$.
Since $p$ is topologically nilpotent, after replacing $n$ by a sufficiently large integer, we obtain $a_g^{p^n} \in J$.
This shows that $a_g$ is topologically nilpotent.
Furthermore, by the proof of (1), we have
\[
\frac{a^p - \sum_{g \in G} a_g^p}{p}
= \lim_{I} b_I.
\]
By construction, each $b_I$ is a sum of topologically nilpotent elements and is therefore itself topologically nilpotent for every $I \in \mathcal{B}_R$.
Hence, the limit $\lim_I b_I$ is also topologically nilpotent, as desired.

Finally, assume that \(a\) is \(p\)-torsion to prove (3).
Since \(a\) is annihilated by \(p\), we have
\[
0 = p a = p \sum_{g \in G} a_g = \sum_{g \in G} p a_g,
\]
and hence \(p a_g = 0\) for every \(g \in G\).
Thus \(a_I^p = a_I^{(p)}\) for every \(I \in \mathcal{B}_R\),
and taking the limit gives \(a^p = \sum_{g \in G} a_g^p\), as claimed.
\end{proof}


In what follows, we give equivalent conditions of (pro-)graded rings being (pro-)graded perfectoid rings.
One of the conditions is a graded variant of \Cref{EquivPerfectoid}.

\begin{proposition}\label{graded-perfectoid-equiv}
Let \((R, R_{\graded})\) be a pro-\(G\)-graded adic ring such that \(p\) is topologically nilpotent in \(R_{\graded}\).
Then the following conditions are equivalent:
\begin{enumerate}
    \item \((R, R_{\graded})\) is pro-\(G\)-graded perfectoid;
    \item The \(p\)-adically complete rings \(R\) and $R_0$ are perfectoid;
    \item there exists a topologically nilpotent element \(\varpi \in R_0\) with \(p \in \varpi^p R_0\) such that the Frobenius map
    \begin{equation} \label{IsomPPower}
        R_{\graded}/\varpi R_{\graded}
        \xrightarrow{\,a \mapsto a^p\,}
        R_{\graded}/\varpi^p R_{\graded}
    \end{equation}
    is an isomorphism, and the multiplicative map
    \begin{equation} \label{MultiplicativePPower}
        R_{\graded}[\varpi^\infty]
        \xrightarrow{\,a \mapsto a^p\,}
        R_{\graded}[\varpi^\infty]
    \end{equation}
    is bijective. 
\end{enumerate}
Furthermore, in either case, we obtain $G=G[1/p]$.
\end{proposition}

\begin{proof}
First, we prove $(1) \Rightarrow (2)$.
By \cref{BoundedTorsionPerfd}(3), $R$ is perfectoid.
Hence there exist $\varpi \in R$ and a unit $u \in R$ such that $p=\varpi^p u$.
Moreover, there exist $x,y \in R$ such that $u=x^p+py$.
Therefore,
\[
p=\varpi^p x^p+p\varpi^p y
 = a^p+p^2 u^{-1} y
\]
holds in \(R\), where we set $a\defeq \varpi x$.
In particular, $a$ is topologically nilpotent and so is \(a_0\) by \Cref{p-th-power}.

Taking the degree-$0$ part and using \cref{p-th-power}(1), we obtain
\[
p=a_0^p+pb_0+p^2(u^{-1}y)_0,
\]
where $b\defeq (a^p-\sum_{g \in G} a_g^p)/p$.
By \cref{p-th-power}(3), the element $b_0$ is topologically nilpotent.
Hence $1-b_0-p(u^{-1}y)_0$ is a unit, and therefore
\[
p(1-b_0-p(u^{-1}y)_0)^{-1}=a_0^p.
\]
Replacing $\varpi$ with $a_0$, we may assume that $\varpi \in R_0$.

Next, to prove that \(R_0\) is perfectoid, we prove that the maps \eqref{IsomPPower} and \eqref{MultiplicativePPower} are bijective.
Since $\varpi$ is topologically nilpotent in $R_0$, the ring $R$ is $\varpi$-adically complete.
Because $R$ is perfectoid and satisfies $p \in \varpi^p R$, \Cref{ChoiceVarpi} implies that the $p$-power maps
\[
R/\varpi R \xrightarrow{\,a \mapsto a^p\,} R/\varpi^p R
\quad \text{and} \quad
R[\varpi^\infty] \xrightarrow{\,a \mapsto a^p\,} R[\varpi^\infty]
\]
are isomorphisms.

We have a commutative diagram
\[
\begin{tikzcd}
R/\varpi R \arrow[r, "a \mapsto a^p", "\cong"'] & R/\varpi^p R \\
R_{\graded}/\varpi R_{\graded} \arrow[u, hook]
  \arrow[r, "a \mapsto a^p"]
& R_{\graded}/\varpi^p R_{\graded} \arrow[u, hook]
\end{tikzcd}
\]
whose vertical arrows are injective by \cref{purity-homog-ideal}.
Hence the induced map
\[
R_{\graded}/\varpi R_{\graded}
  \xrightarrow{\,a \mapsto a^p\,}
R_{\graded}/\varpi^p R_{\graded}
\]
is injective.

To prove surjectivity, let $a \in R_{\graded}$ be a homogeneous element of degree $h \in G$.
By \cref{p-th-power}, there exists $b=\sum_{g \in G} b_g \in R$ such that
\[
a \equiv b^p \equiv \sum_{g \in G} b_g^p
\pmod{\varpi^p R}.
\]
Comparing homogeneous components and using \cref{purity-homog-ideal}, we obtain $h/p \in G$ and
\[
a \equiv b_{h/p}^p
\pmod{\varpi^p R}.
\]
Since $b_{h/p} \in R_{\graded}$ and the map \eqref{IsomPPower} is graded, we conclude that
\[
R_{\graded}/\varpi R_{\graded}
\xrightarrow{\sim}
R_{\graded}/\varpi^p R_{\graded}
\]
is an isomorphism of $G$-graded rings.
In particular, this implies $G=G[1/p]$.

Moreover, since the maps \eqref{IsomPPower} and \eqref{MultiplicativePPower} are graded by \cref{p-th-power}, taking degree-$0$ parts yields bijections
\[
R_0/pR_0
\xrightarrow{\,a \mapsto a^p\,}
R_0/pR_0,
\qquad
R_0[p]
\xrightarrow{\,a \mapsto a^p\,}
R_0[p].
\]
By \cref{EquivPerfectoid}, the ring $R_0$ is perfectoid.
Thus (2) holds.

The implication $(2) \Rightarrow (3)$ follows from the same argument as in the proof of $(1) \Rightarrow (2)$.

Finally, we prove $(3) \Rightarrow (1)$.
Set $R'\defeq R_{\graded}^{\wedge p}$.
By \Cref{torsion-compare}, the bijectivity of \eqref{MultiplicativePPower} implies that $R_{\graded}$ has bounded $\varpi$-torsion.
By \cite[Lemma~3.16]{ishiro2025Perfectoida}, bounded $\varpi^\infty$-torsion in $R_{\graded}$ implies
\[
R'[\varpi^\infty]=R_{\graded}[\varpi^\infty].
\]
Hence the multiplicative $p$-power map on $R'[\varpi^\infty]$ is bijective.
Moreover, since
\[
R'/\varpi R' \simeq R/\varpi R,
\qquad
R'/\varpi^p R' \simeq R/\varpi^p R,
\]
we obtain an isomorphism
\[
R'/\varpi R'
\xrightarrow{\,a \mapsto a^p\,}
R'/\varpi^p R'.
\]
Therefore $R'$ is perfectoid by \Cref{EquivPerfectoid}, and thus (1) holds.

Furthermore, by the proof of $(1) \Rightarrow (2)$, we also obtain $G=G[1/p]$, as claimed.
\end{proof}

\begin{corollary}\label{except-for-p-comp}
Let \(R\) be a $p$-adically gradedwise complete \(G\)-graded ring.
Then the following conditions are equivalent:
\begin{enumerate}
    \item The ring \(R\) is graded perfectoid;
    \item the \(p\)-adic completion \(R^{\wedge p}\) and the degree-\(0\) part \(R_0\) are perfectoid;
    \item there exists a topologically nilpotent  element \(\varpi \in R_0\) with \(p \in \varpi^p R_0\) such that the Frobenius map
    \[
    R/\varpi R \xrightarrow{\,a \mapsto a^p\,} R/\varpi^p R
    \]
    is an isomorphism, and the multiplicative map
    \[
    R[\varpi^\infty] \xrightarrow{\,a \mapsto a^p\,} R[\varpi^\infty]
    \]
    is bijective.
\end{enumerate}
Furthermore, in either case, we have $G=G[1/p]$.
\end{corollary}

\begin{proof}
Equip \(R\) with the \(p\)-adic topology.
Then the pair \((R^{\wedge p}, R)\) forms a pro-\(G\)-graded \(p\)-adic ring.
Hence the assertion follows from \cref{graded-perfectoid-equiv}.
\end{proof}

\begin{remark} \label{RelationGR}
\Cref{graded-perfectoid-equiv} essentially follows from the proof of \cite[Proposition 16.6.1]{gabber2018Foundations}. For the convenience of the reader, we rewrite the argument in our terminologies.
\end{remark}

\begin{corollary} \label{PureSubPerfd}
Let \((R, R_{\graded})\) be a pro-\(G\)-graded perfectoid ring with a homogeneous ideal of definition \(I\).
Let \(H \subseteq G\) be a submonoid satisfying \(H = H[1/p]\), and set 
\[
R'_{\graded} \coloneqq \bigoplus_{h \in H} R_h \subseteq R_{\graded}.
\]
Then the pair \((R', R'_{\graded})\) is an \(H\)-graded perfectoid ring, where 
\(R'\) denotes the \((I \cap R'_{\graded})\)-adic completion of \(R'_{\graded}\).
\end{corollary}

\begin{proof}
Since \(R'_{\graded}\) is gradedwise complete, the pair \((R', R'_{\graded})\) forms a pro-\(H\)-graded ring.
Because \((R, R_{\graded})\) is \(G\)-graded perfectoid, there exists an element 
\(\varpi \in R_0\) such that \(p \in \varpi^p R_{\graded}\), and \(R_{\graded}\) satisfies the conditions of \cref{graded-perfectoid-equiv}.
In particular, by \cref{torsion-compare} we have
\(R_{\graded}[\varpi] = R_{\graded}[\varpi^\infty]\),
and the same equality holds for \(R'_{\graded}\).

These properties imply that \(R'_{\graded}\) admits an element \(\varpi \in R_0\) satisfying \(p \in \varpi^p R'_{\graded}\),
and that the graded homomorphisms
\[
R'_{\graded}/\varpi R'_{\graded} 
    \xrightarrow{\,a \mapsto a^p\,} 
    (R'_{\graded}/\varpi^p R'_{\graded})^{[p]},
    \qquad
R'_{\graded}[\varpi] 
    \xrightarrow{\,a \mapsto a^p\,} 
    (R'_{\graded}[\varpi])^{[p]}
\]
are isomorphisms.
Here the assumption \(H = H[1/p]\) ensures that the grading is preserved under the \(p\)-th power map.

Since \(R'_{\graded}[\varpi] = R'_{\graded}[\varpi^\infty]\),
it follows from \cref{graded-perfectoid-equiv} that the pro-\(H\)-graded ring \((R', R'_{\graded})\) is \(H\)-graded perfectoid, as desired.
\end{proof}

Finally, we will show that graded perfectoid rings have the similar properties for non-graded perfectoid rings
The following twe lemmas are on adic completions and bounded torsions in graded perfectoid rings as in \Cref{BoundedTorsionPerfd}.

\begin{lemma} \label{CompatibleSystemModpGraded}
    Let \(R\) be a \(G\)-graded ring such that \(R\) is \(p\)-adically gradedwise complete and \(R/pR\) is semiperfect.
    Take a homogeneous element \(f\) of \(R\).
    Then there exists a homogeneous element \(g\) of \(R\) such that \(f \equiv g \bmod pR\) and \(g\) has a compatible system \(\{g^{1/p^n}\}_{n \geq 0}\) of homogeneous \(p\)-power roots of \(g\) in \(R\).
\end{lemma}

\begin{proof}
    This proof is the same as \Cref{CompatibleSystemModp} under taking care of the grading.
\end{proof}

\begin{lemma} \label{BoundedTorsionGradedPerfd}
    Let \((R, R_{gr})\) be a \(p\)-adic \(G\)-graded perfectoid ring and let \(I\) be a finitely generated homogeneous ideal of \(R_{\graded}\) containing \(p\).
    Then the following assertions holds.
    \begin{enumerate}
        \item There exists a homogeneous generator \(f_1, \dots, f_r\) of \(I\) such that \(f_i\) has a compatible system of \(p\)-power roots in \(R_{\graded}\).
        \item The \(G\)-graded ring \(R_{\graded}\) has bounded \(f_i^{\infty}\)-torsion for such choice of \(f_i\) in (1). In particular, \(R_{\graded}\) has bounded \(I^{\infty}\)-torsion.
        \item The pair \((\comp{I}{R}, \grcomp{I}{R_{\graded}})\) is an \(I\)-adic pro-\(G\)-graded perfectoid ring.
        \item The canonical morphism
        \begin{equation*}
            \dcomp{I}{R_{\graded}} \to \comp{I}{R_{\graded}}
        \end{equation*}
        is an isomorphism in \(\mcalD(R_{\graded})\), where \(\dcomp{I}{R_{\graded}}\) is the derived \(I\)-completion of \(R_{\graded}\).
    \end{enumerate}
\end{lemma}

\begin{proof}
    (1): This follows from \Cref{CompatibleSystemModpGraded}.
    
    (2): Since \(R_{gr} \hookrightarrow R\) is injective, this follows from \Cref{BoundedTorsionPerfd}(2).
    
    (3): The pair \((\comp{I}{R}, \grcomp{I}{R_{\graded}})\) is a pro-\(G\)-graded adic ring.
    Since \(\comp{I}{R}\) is a perfectoid ring by \Cref{BoundedTorsionPerfd}(3), this pair is an \(I\)-adic pro-\(G\)-graded ring by \Cref{graded-perfectoid-equiv}(2).

    (4): The morphism can be factorized to
    \begin{equation*}
        \dcomp{I}{R_{gr}} \cong \dcomp{I}{\dcomp{p}{R_{gr}}} \xrightarrow{\cong} \dcomp{I}{\comp{p}{R_{gr}}} \xrightarrow{\cong} \comp{I}{\comp{p}{R_{gr}}} = \comp{I}{R_{gr}},
    \end{equation*}
    where the second isomorphism follows from the boundedness of \(p^{\infty}\)-torsions in \(R_{gr}\) and the third one follows from \Cref{BoundedTorsionPerfd}(3) since \(\comp{p}{R_{gr}}\) is a perfectoid ring by \Cref{DefGradedPerfectoid}.
\end{proof}

Perfectoid rings have a canonical decomposition into the \(p\)-torsion-free part and the characteristic \(p\) part as in \cite[Remark 3.9]{bhatt2019Regular}.

\begin{proposition} \label{DecompositionPerfectoid}
  Let \(R\) be a \(G\)-graded perfectoid ring with a homogeneous ideal of definition \(I\) containing \(p\).
  Then we have a canonical isomorphism of \(G\)-graded rings
  \begin{equation*}
    R \xrightarrow{\cong} R_{\tf} \times_{(R_{\tf}/pR_{\tf})_{\perf}} (R/pR)_{\perf},
  \end{equation*}
  where \(R_{\tf} \defeq R/R[p^\infty]\) is the quotient of \(R\) by its \(p^{\infty}\)-torsion part and \((-)_{\perf}\) denotes the perfection.
\end{proposition}

\begin{proof}
  We first show that the natural map \(R_{\tf} \to (R^{\wedge p})_{\tf}\) is injective and can be identified with the \(p\)-adic completion of \(R_{\tf}\): As in the proof of \Cref{graded-perfectoid-equiv}, the equalities \(R[p^\infty] = R[p] = R^{\wedge p}[p] = R^{\wedge p}[p^\infty]\) hold and this shows the injectivity of the map \(R_{\tf} \to (R^{\wedge p})_{\tf}\).
  Taking the derived \(p\)-completion of the exact sequence \(0 \to R[p] \to R \to R_{\tf} \to 0\) of bounded \(p^\infty\)-torsion objects, we obtain an exact sequence
  \begin{equation*}
    0 \to R[p] \to R^{\wedge p} \to (R_{\tf})^{\wedge p} \to 0.
  \end{equation*}
  So we have \((R^{\wedge p})_{\tf} \cong (R_{\tf})^{\wedge p}\).

  Therefore, canonical morphisms gives a commutative diagram of rings
  \begin{center}
    \begin{tikzcd}
      R \arrow[d, hook] \arrow[r]     & R_{\tf} \times_{(R_{\tf}/pR_{\tf})_{\perf}} (R/pR)_{\perf} \arrow[d, hook] \\
      R^{\wedge p} \arrow[r, "\cong"] & (R^{\wedge p})_{\tf} \times_{(R_{\tf}/pR_{\tf})_{\perf}} (R/pR)_{\perf}   
    \end{tikzcd}
  \end{center}
  where the lower horizontal map is an isomorphism by \cite[Remark 3.9]{bhatt2019Regular} and the upper horizontal map is a morphism of \(G\)-graded rings, especially, injective.
  Considering the topological homogeneous decomposition of the inverse image of the image of an element in the right upper corner, we can see the surjectivity of the upper horizontal map.
  Then we obtain the desired isomorphism.
\end{proof}

\subsection{Graded Andr\'e's flatness lemma}

We will prove the graded variant of Andr\'e's flatness lemma.
To prove this, we use the notion of the \emph{\(p\)-root closure} \(C(R)\) of a \(p\)-torsion-free ring \(R\), which is the subset \(\set{x \in R[1/p]}{\exists n \in \setZ_{\geq 1}, \ x^{p^n} \in R}\).
Actually, this becomes a subring of \(R[1/p]\) following \cite{roberts2008Root}. See \cite{ishizuka2024Calculationa} for some history of root closure and its relationship to perfectoid theory.

In the first-named author's paper \cite{ishizuka2024Calculationa}, there was a non-trivial argument in the proof of main theorem.
Here, we fulfill the gap under a private communication with L\'eo Navarro Chafloque.

\begin{proposition}\label{perfectoidization-p-root}\textup{(cf. \cite[Theorem 1.3 and Corollary 5.9]{ishizuka2024Calculationa})}
    Let \(R\) be a \(p\)-torsion-free ring such that whose \(p\)-adic completion \(\widehat{R}\) is a semiperfectoid ring.
    Then \(R_{\perfd}\) is isomorphic to the \(p\)-adic completion \(C(R)^{\wedge}\) of the \(p\)-root closure \(C(R)\) of \(R\) in \(R[1/p]\).
\end{proposition}

\begin{proof}
    First note that since \(\widehat{R}\) is a semiperfectoid ring, it has a compatible system of \(p\)-power roots \(\{\varpi^{1/p^n}\}_{n \geq 0}\) of \(\varpi\) such that \(p\) is some unit multiple of \(\varpi^p\) in \(\widehat{R}\).
    In this case, we can deduce that \((\widehat{R})_{\perfd}\) is isomorphic to \(C(\widehat{R})^{\wedge}\) by \cite[Theorem 1.3 and Corollary 5.9]{ishizuka2024Calculationa}.
    Moreover, since any perfectoid rings are \(p\)-adically complete, \((\widehat{R})_{\perfd}\) is the initial perfectoid ring over \(R\) and therefore we have \(R_{\perfd} \cong C(\widehat{R})^{\wedge}\).
    On the other hand, we can show that both \(C(R)^{\wedge}\) and \(C(\widehat{R})^{\wedge}\) are the initial object in the full subcategory of \(R\)-algebras such that \(p\)-torsion-free, \(p\)-adically complete, and \(p\)-root closed.
    Note that \(p\)-root closedness of a \(p\)-torsion-free ring \(A\) is equivalent to that if \(a \in A\) satisfies \(a^p = 0\) in \(A/p^pA\), then \(a = 0\) in \(A/pA\). Especially the \(p\)-root closed property is stable under taking \(p\)-adic completion.
\end{proof}

\begin{corollary} \label{perfectoidization-general}
    Let \(R\) be a ring such that whose \(p\)-adic completion \(\widehat{R}\) is a semiperfectoid ring and let \(R' \defeq R/R[p^\infty]\) be the quotient of \(R\) by its \(p^{\infty}\)-torsion part.
    Then the perfectoidization \(R \to R_{\perfd}\) can be identified with
    \begin{equation*}
        R \to C(R')^{\wedge} \times_{(C(R')/pC(R'))_{\perf}} (R/pR)_{\perf},
    \end{equation*}
    where \(C(R')\) is the \(p\)-root closure of the \(p\)-torsion-free ring \(R'\) in \(R'[1/p]\) and \((-)_{\perf}\) denotes the perfection. 
\end{corollary}

\begin{proof}
    This follows from \Cref{perfectoidization-p-root} and \Cref{DecompositionPerfectoidization} since \(R'\) is a \(p\)-torsion-free ring whose \(p\)-adic completion is semiperfectoid.
\end{proof}

\begin{lemma} \label{Graded-pRootClosure}
We assume $G$ is totally ordered which commutes with additions.
Let \(R_{\graded}\) be a \(G\)-graded ring. Then there exists a \(G\)-grading on the \(p\)-root closure \(C(R_{\graded})\) of \(R_{\graded}\) in \(R_{\graded}[1/p]\) such that \(R_{\graded} \hookrightarrow C(R_{\graded}) \hookrightarrow R_{\graded}[1/p]\) is a sequence of \(G\)-graded morphisms.
\end{lemma}

\begin{proof}
    We prove that the $p$-root closure $C(R_{\graded})$ of $R_{\graded}$ is the $G$-graded subring of $R_{\graded}[1/p]$.
    Take $a \in C(R_{\graded})$ and the homogeneous decomposition $a=\sum a_g$ in $R_{\graded}[1/p]$. 
    We prove $a_g \in C(R_{\graded})$ for every $g \in G$ by induction on the number of the set $G_a\defeq\{g \in G \mid a_g \neq 0\}$.
    There exists an integer $n \geq 1$ such that $a^{p^n} \in R_{\graded}$.
    Take the maximum element $g_0$ in $G_a$, then the degree-$p^ng_0$ part of $a^{p^n}$ is $a_{g_0}^{p^n}$, thus $a_{g_0}^{p^n} \in R_{\graded}$.
    In conclusion, we have $a_{g_0} \in C(R_{\graded})$.
    Therefore, we have $a'\defeq a-a_{g_0} \in C(R_{\graded})$ and $\# G_a =\# G_{a'}+1$.
    By the induction hypothesis, we have $a_g \in C(R_{\graded})$ for every $g \in G$, as desired.
\end{proof}


In semiperfectoid case, the universal perfectoid ring exists by \cite{bhatt2022Prismsa}.
The same story holds for the graded case (\Cref{perfectoidization-graded}).
To show this, we define the notion of graded semiperfectoid rings:

\begin{definition} \label{DefSemiperfectoidGraded}
We define a graded variant of semiperfectoid rings.
\begin{itemize}
    \item Let $R$ be a $G$-graded ring.
    We say that $R$ is \emph{graded semiperfectoid} if there exist a $G$-graded perfectoid $P$ and a graded surjective ring homomorphism $P \to R$. 
    \item Let $(R,R_{\graded})$ be a pro-\(G\)-graded adic ring.
    We say that $(R,R_{\graded})$ is a \emph{pro-$G$-graded semiperfectoid ring} if there exists a \(G\)-graded perfectoid ring \((P, P_{\graded})\) and a $G$-graded ring homomorphism \((P, P_{\graded}) \to (R, R_{\graded})\) such that $P \to R$ is surjective.
\end{itemize}
\end{definition}

\begin{remark} \label{RemarkSemiperfd}
    By \cref{SurjectiveProGraded}, for a pro-\(G\)-graded semiperfectoid ring \((R, R_{\graded})\) above, the homomorphism $P_{\graded} \to R_{\graded}$ is also surjective, and especially, \(R_0\) is a semiperfectoid ring and $R_{\graded}$ is graded semiperfectoid.
\end{remark}

\begin{remark} \label{BSPerfdGradedRings}
    Take a \(G\)-graded ring \(R\) and its \(p\)-adic completion \(R^{\wedge p}\).
    Assume that \(R\) admits a graded ring homomorphism \(P \to R\) from a \(G\)-graded perfectoid ring \(P\).
    Then \(R\) (resp., \(R^{\wedge p}\)) has a ring homomorphism \(P_0 \to R\) (resp., \(P^{\wedge p} \to R^{\wedge p}\)) from a perfectoid ring \(P_0\) (resp., \(P^{\wedge p}\)) by the definition of graded perfectoid rings (\Cref{DefGradedPerfectoid}).
    Then we can take the \emph{perfectoidization} \(R_{\perfd}\) (resp., \((R^{\wedge p})_{\perfd}\)) of \(R\) (resp., \(R^{\wedge p}\)) in the sense of \cite{bhatt2022Prismsa} and they are the same thing because of \cite[Proposition 8.5]{bhatt2022Prismsa}.
    Especially, if \(R\) is a graded semiperfectoid ring, the perfectoidization \(R_{\perfd} \cong (R^{\wedge p})_{\perfd}\) is the initial perfectoid \(R\)-algebra by \cite[Corollary 7.3]{bhatt2022Prismsa} and has an explicit representation (\Cref{perfectoidization-general}).

    In the case of pro-\(G\)-graded ring \((R, R_{\graded})\), the same story holds: If \((R, R_{\graded})\) has a morphism from a pro-\(G\)-graded perfectoid ring \((P, P_{\graded})\), we can take the perfectoidization \(R_{\perfd}\) and \((R_{\graded})_{\perfd}\) of \(R\) and \(R_{\graded}\) in the sense of \cite{bhatt2022Prismsa} respectively.
    Assuming that \((R, R_{\graded})\) is semiperfectoid, then we can show that \(R_{\perfd}\) (resp., \((R_{\graded})_{\perfd}\)) is the initial (\(p\)-adic) perfectoid \(R\)-algebra (resp., \(R_{\graded}\)-algebra).
\end{remark}

A priori, even if a graded ring \(R\) admitting a morphism from a graded perfectoid ring has \(R_{\perfd}\) which is concentrated in degree \(0\), this \(R_{\perfd}\) may not have a suitable graded structure.
However, if we consider graded semiperfectoid rings, the next theorem states that a natural graded structure exists on it in the pro-graded sense.

\begin{theorem} \label{perfectoidization-graded}
We assume $G$ is totally ordered which commutes with additions.
Let \((R,R_{\graded})\) be a \(G\)-graded semiperfectoid ring.
Then there exists a \(G\)-graded perfectoid ring \((R, R_{\graded})_{\perfd}\) over \((R, R_{\graded})\) which is the initial \(G\)-graded perfectoid ring over \((R, R_{\graded})\). We call this the \emph{perfectoidization} of \((R, R_{\graded})\).

Explicitly, the complete part of \((R, R_{\graded})_{\perfd}\) is the \(I\)-adic completion of the (\(p\)-adic) perfectoidization \(R_{\perfd}\) of \(R\), where \(I\) is a homogeneous ideal of definition of \(R_{\graded}\) containing \(p\).
If \(R\) is \(p\)-torsion free, it is given by the pair \((C(R_{\graded})^{\wedge I}, C(R_{\graded})^{\grwedge})\) where \(C(R_{\graded})\) is the \(p\)-root closure of \(R_{\graded}\) in \(R_{\graded}[1/p]\) following \Cref{Graded-pRootClosure}.
\end{theorem}

\begin{proof}
Fix a homogeneous ideal of definition \(I\) of \(R_{\graded}\) and a surjective morphism \((P, P_{\graded}) \to (R, R_{\graded})\) from a \(G\)-graded perfectoid ring \((P, P_{\graded})\). Since \(p\) is topologically nilpotent in \(P_{\graded}\), so is in \(R_{\graded}\) and then we can assume that \(I\) contains \(p\).
Take the \(p\)-torsion-free quotient \(R'_{\graded} \defeq R_{\graded}/R_{\graded}[p^\infty]\). Since \(R_{\graded}[p^{\infty}]\) is the kernel of the graded morphism \(R_{\graded} \to R_{\graded}[1/p]\), it is a homogeneous ideal of \(R_{\graded}\) and then \(R'_{\graded}\) is a \(G\)-graded semiperfectoid ring.

By \Cref{BSPerfdGradedRings} and \Cref{perfectoidization-general}, the \(G\)-graded semiperfectoid ring \(R_{\graded}\) admits the perfectoidization \((R_{\graded})_{\perfd}\), which is isomorphic to the one \((R^{\wedge p})_{\perfd}\) of the semiperfectoid ring \(R^{\wedge p}\) and it can be written as
\begin{align} \label{PerfectoidizationGradedEq}
    (R_{\graded})_{\perfd} & \cong C(R'_{\graded})^{\wedge p} \times_{(C(R'_{\graded})/pC(R'_{\graded}))_{\perf}} (R_{\graded}/pR_{\graded})_{\perf} \\
    & \cong (C(R'_{\graded}) \times_{(C(R'_{\graded})/pC(R'_{\graded}))_{\perf}} (R_{\graded}/pR_{\graded})_{\perf})^{\wedge p} \nonumber
\end{align}
as \(R\)-algebras.
We can equip \(C(R'_{\graded})\) with a \(G\)-graded \(R'_{\graded}\)-algebra structure by \Cref{Graded-pRootClosure}. Moreover, since the Frobenius morphism on \(G\)-graded \(\setF_p\)-algebra can be seen as a \(G\)-graded morphism as in \Cref{graded-witt-tilt}(1), the perfection \((-)_{\perf}\) of a \(G\)-graded ring has a natural \(G\)-grading.
Set an \(R\)-algebra
\begin{equation} \label{GradedStructurePerfd}
    R_{\perfd, \graded} \defeq C(R'_{\graded})^{\grwedge} \times_{(C(R'_{\graded})/pC(R'_{\graded}))_{\perf}} (R_{\graded}/pR_{\graded})_{\perf},
\end{equation}
where \(C(R'_{\perfd})^{\grwedge}\) is the \(p\)-adic gradedwise completion.
Then \(R_{\perfd, \graded}\) is the pullback of \(p\)-adically gradedwise complete \(G\)-graded rings along \(G\)-graded morphisms, we can equipped with a natural \(G\)-graded \(R_{\graded}\)-algebra structure, which is again \(p\)-adically gradedwise complete.
Using this \(G\)-grading, \Cref{except-for-p-comp} says that \(R_{\perfd, \graded}\) is a \(G\)-graded perfectoid ring and then we can take a pro-\(G\)-graded \(p\)-adic perfectoid ring \(((R_{\graded})_{\perfd}, R_{\perfd, \graded})\) over \((R_{\graded}^{\wedge p}, R_{\graded})\).

We will show that this \(((R_{\graded})_{\perfd}, R_{\perfd, \graded})\) is the initial pro-\(G\)-graded \(p\)-adic perfectoid ring over \((R_{\graded}^{\wedge p}, R_{\graded})\).
Take any pro-\(G\)-graded \(p\)-adic perfectoid ring \((P, P_{\graded})\) over \((R_{\graded}^{\wedge p}, R_{\graded})\).
By construction, \(R_{\graded}^{\wedge p} \to P\) uniquely factors through the perfectoidization \((R_{\graded})_{\perfd}\).
Since \(R_{\graded}^{\wedge p}\) is a semiperfectoid ring, the morphism \(R_{\graded}^{\wedge p} \to (R_{\graded})_{\perfd}\) is surjective (\cite[Theorem 7.4]{bhatt2022Prismsa}).
Together with \Cref{SurjectiveProGraded}, we have a commutative diagram
\begin{center}
  \begin{tikzcd}
    R_{\graded} \arrow[dd, hook] \arrow[rd, two heads] \arrow[rr] &                                                & P_{\graded} \arrow[dd, hook] \\
                                                                  & {R_{\perfd, \graded}} \arrow[dd, hook]         &                              \\
    R_{\graded}^{\wedge p} \arrow[rr] \arrow[rd, two heads]       &                                                & P                            \\
                                                                  & (R_{\graded})_{\perfd} \arrow[ru, "\exists !"] &                             
  \end{tikzcd}
\end{center}
of ring homomorphisms.
The morphism \((R_{\graded})_{\perfd} \to P\) uniquely induces a graded morphism \(R_{\perfd, \graded} \to P_{\graded}\): Take any homogeneous element \(x \in R_{\perfd, \graded}\) with a homogeneous lift \(\widetilde{x} \in R_{\graded}\).
The image \(y \in P_{\graded}\) of \(\widetilde{x}\) is uniquely determined since its image in \(P\) is determined by the image of \(x\) along \(R_{\perfd, \graded} \hookrightarrow (R_{\graded})_{\perfd} \to P\).
Therefore, we can get the unique morphism \(((R_{\graded})_{\perfd}, R_{\perfd, \graded}) \to (P, P_{\graded})\) of pro-\(G\)-graded \(p\)-adic perfectoid rings over \((R_{\graded}^{\wedge p}, R_{\graded})\).
So this is the initial pro-\(G\)-graded \(p\)-adic perfectoid ring over \((R_{\graded}^{\wedge p}, R_{\graded})\).

By \Cref{BoundedTorsionGradedPerfd}(3), the \(I\)-adic completion \((\comp{I}{(R_{\graded})_{\perfd}}, \grcomp{I}{R_{\perfd, \graded}})\) of \(((R_{\graded})_{\perfd}, R_{\perfd, \graded})\) is the initial pro-\(G\)-graded \(I\)-adic perfectoid ring over \((R, R_{\graded})\).
Therefore, when we set \((R, R_{\graded})_{\perfd}\) as this, we obtain a pro-$G$-graded ring \((R, R_{\graded})_{\perfd}\) and this is the initial \(G\)-graded perfectoid ring over \((R, R_{\graded})\).

Since \(\comp{I}{(R_{\graded})_{\perfd}}\) is the initial \(I\)-adic complete perfectoid \(R\)-algebra, this is isomorphic to the \(I\)-adic completion \((R_{\perfd})^{\wedge I}\) of the perfectoidization \(R_{\perfd}\) of \(R\).
Moreover, if \(R\) is \(p\)-torsion free, then \((R_{\graded})_{\perfd}\) (resp., \(R_{\perfd, \graded}\)) is isomorphic to \(C(R_{\graded})^{\wedge p}\) (resp., \(C(R_{\graded})\)) by \eqref{PerfectoidizationGradedEq} (resp., \eqref{GradedStructurePerfd}), and then \((R, R_{\graded})_{\perfd}\) can be given by \((C(R_{\graded})^{\wedge I}, C(R_{\graded})^{\grwedge})\) as claimed.
\end{proof}

\begin{corollary}\label{gr-perfdion}
Let \(R\) be a \(G\)-graded semiperfectoid ring.
Then there exists a \(G\)-graded perfectoid ring \(R_{\perfd,{\gr}}\) over \(R\),
which is the initial \(G\)-graded perfectoid over \(R\) and its \(p\)-adic completion coincides with \(R_{\perfd}\).
\end{corollary}

\begin{proof}
Since \(R\) is graded semiperfectoid, there exists a \(G\)-graded perfectoid ring \(P\)
and a surjective graded ring homomorphism \(P \to R\).
Equipping \(P\) and \(R\) with the \(p\)-adic topology,
we obtain a surjection
\[
P^{\wedge p} \longrightarrow R^{\wedge p}
\]
from the pro-\(G\)-graded \(p\)-adic perfectoid pair \((P^{\wedge p}, P)\).
In particular, the pair \((R^{\wedge p}, R)\) is a pro-\(G\)-graded semiperfectoid.

By \cref{perfectoidization-graded}, there exists a perfectoidization
\[
(R^{\wedge p}, R)_{\perfd} = (R_{\perfd}, R_{\perfd, \graded})
\]
of \((R^{\wedge p}, R)\) by using (\ref{GradedStructurePerfd}).
Then \(R_{\perfd,\gr}\) is the initial \(G\)-graded perfectoid over \(R\)
by \cref{first-cat-equiv}, as desired.
\end{proof}

Using this existence, we will prove a graded variant of Andr\'e's flatness lemma (\Cref{GradedAndre}).
Before that, we will prepare some lemma.

\begin{lemma} \label{PolynomialRingPerfectoid}
    Let \((R, R_{\graded})\) be a \(G\)-graded perfectoid ring with a homogeneous ideal of definition \(I\) containing \(p\). Assume \(G = G[1/p]\).
    Set a polynomial ring \(R_{\graded}[T]\) with the \(G\)-grading given by \(\deg(T) = h\) for some \(h \in G\).
    Then the pair \((R[T^{1/p^\infty}]^{\wedge I}, R_{\graded}[T^{1/p^\infty}]^{\grwedge})\) is a \(G\)-graded perfectoid ring where \(R[T^{1/p^\infty}]^{\wedge I}\) is the \(I\)-adic completion of \(R[T^{1/p^\infty}]\) and \(R_{\graded}[T^{1/p^\infty}]^{\grwedge}\) is the \(I\)-adic gradedwise completion of \(R[T^{1/p^\infty}]\).
\end{lemma}

\begin{proof}
    First of all, the pair \((R[T^{1/p^\infty}]^{\wedge I}, R_{\graded}[T^{1/p^\infty}]^{\grwedge})\) forms a pro-\(G\)-graded ring.
    Since \(R\) is a perfectoid ring by \Cref{graded-perfectoid-equiv}, the \(p\)-adic completion \(R[T^{1/p^\infty}]^{\wedge p}\) is a perfectoid ring as usual.
    Then \Cref{BoundedTorsionPerfd}(3) shows that the \(I\)-adic completion \(R[T^{1/p^\infty}]^{\wedge I}\) is also a perfectoid ring.
    Moreover, again \Cref{graded-perfectoid-equiv} implies that there exists a topologically nilpotent \(\varpi\) of \(R_0\) such that \(p \in \varpi^pR_0\).
    Since the degree-\(0\) part of \(R_{\graded}[T^{1/p^\infty}]\) contains \(R_0\), this element \(\varpi\) satisfies the same condition for \(R_{\graded}[T^{1/p^\infty}]^{\grwedge}\). Then the pair \((R[T^{1/p^\infty}]^{\wedge I}, R_{\graded}[T^{1/p^\infty}]^{\grwedge})\) satisfies all the conditions in \Cref{graded-perfectoid-equiv}.
    \qedhere
\end{proof}

\begin{definition} \label{DefCompletelyFFlat}
    A morphism \((\varphi, \varphi_{\graded}) \colon (R, R_{\graded}) \to (S, S_{\graded})\) of pro-\(G\)-graded adic rings is called \emph{completely faithfully flat} if the graded morphism \(\varphi_{\graded} \colon R_{\graded} \to S_{\graded}\) is an adic\footnote{A morphism of adic rings \(f \colon A \to B\) is called \emph{adic} if \(f(I)B\) becomes an ideal of definition of \(B\) for some (or equivalently any) ideal definition of \(A\).} and \(I\)-completely faithfully flat morphism for some (or equivalently any) homogeneous ideal of definition \(I\) of \(R_{\graded}\).
\end{definition}

\begin{lemma} \label{EquivCFFlat}
    Take a morphism \((\varphi, \varphi_{\graded}) \colon (R, R_{\graded}) \to (S, S_{\graded})\) of pro-\(G\)-graded perfectoid rings and take a homogeneous ideal of definition \(I\) of \(R_{\graded}\).
    Assume that \(\varphi_{\graded}\) is an adic morphism of \(G\)-graded adic rings.
    Then \(\varphi_{\graded} \colon R_{\graded} \to S_{\graded}\) is \(I\)-completely faithfully flat if and only if the morphism \(\varphi \colon R \to S\) is \(IR\)-completely faithfully flat.
\end{lemma}

\begin{proof}
    Since \(\varphi_{\graded}\) is a \(G\)-graded adic morphism, the extension \(J \defeq \varphi_{\graded}(I)S_{\graded}\) is a homogeneous ideal of definition of \(S_{\graded}\).
    Because of the isomorphism \(R_{\graded}/I \cong R/IR\), it suffices to show that \(S \otimes^L_R R/IR\) is isomorphic to \(S_{\graded} \otimes^L_{R_{\graded}} R_{\graded}/I\).
    Since \((R, R_{\graded})\) and \((S, S_{\graded})\) are pro-\(G\)-graded perfectoid rings, especially they are \(p\)-adic pro-\(G\)-graded perfectoid rings.
    Then \Cref{BoundedTorsionGradedPerfd}(4) gives isomorphisms
    \begin{equation*}
        \dcomp{I}{R_{\graded}} \xrightarrow{\cong} \comp{I}{R_{\graded}} = R \quad \text{and} \quad \dcomp{I}{S_{\graded}} = \dcomp{J}{S_{\graded}} \xrightarrow{\cong} \comp{J}{S_{\graded}} = S.
    \end{equation*}
    Using this, we can show that
    \begin{equation*}
        S \otimes^L_R R/IR \cong \dcomp{I}{S_{\graded}} \otimes^L_{\dcomp{I}{R_{\graded}}} R_{\graded}/I \cong \dcomp{I}{S_{\graded} \otimes^L_{R_{\graded}} R_{\graded}/I} \cong S_{\graded} \otimes^L_{R_{\graded}} R_{\graded}/I,
    \end{equation*}
    where the second and last isomorphism follows from that the derived \(I\)-completion is symmetric monoidal and the derived \(I\)-completeness of \(S_{\graded} \otimes^L_{R_{\graded}} R_{\graded}/I\).
\end{proof}

\begin{theorem} \label{GradedAndre}
    Let \((R, R_{\graded})\) be a \(G\)-graded perfectoid ring and let \(\mcalS\) be a set of homogeneous elements in some graded integral extension of \(R_{\graded}\).
    
    Then there exists a completely faithfully flat map \((R, R_{\graded}) \to (R', R'_{\graded})\) of pro-\(G\)-graded perfectoid rings such that \(R'_{\graded}\) admits compatible systems of \(p\)-power roots of elements in \(\mcalS\).
    In particular, there exists a graded absolutely integrally closed\footnote{A pro-\(G\)-graded ring \((R, R_{\graded})\) is called \emph{graded absolutely integrally closed} if there is no non-trivial graded integral extension of \(R_{\graded}\).} \(G\)-graded perfectoid ring which is completely faithfully flat over \((R, R_{\graded})\).
\end{theorem}

\begin{proof}
    Take a homogeneous ideal of definition \(I\) of \(R_{\graded}\) containing \(p\).
    For each element \(s \in \mcalS\), we can find a monic polynomial \(f_s(T) \in R_{\graded}[T]\) and \(g_s \in G\) such that \(f_s(s) = 0\) and \(f_s(T)\) is a homogeneous element in \(R_{\graded}[T]\) if we set \(\deg(T) = g_s\).
    The pair of rings
    \begin{equation*}
        S \defeq \parenlr{R[T_s^{1/p^\infty}]_{s \in \mcalS}/(f_s(T_s))_{s \in \mcalS}}^{\wedge I} \quad \text{and} \quad S_{\graded} \defeq \parenlr{R_{\graded}[T_s^{1/p^\infty}]_{s \in \mcalS}/(f_s(T_s))_{s \in \mcalS}}^{\grwedge}
    \end{equation*}
    defines a \(G\)-graded semiperfectoid rings, i.e., \((S, S_{\graded})\) is a pro-\(G\)-graded ring admitting a surjection from a \(G\)-graded perfectoid ring \((R[T_s^{1/p^\infty}]_{s \in \mcalS}^{\wedge I}, R_{\graded}[T_s^{1/p^\infty}]_{s \in \mcalS}^{\grwedge})\) by \Cref{PolynomialRingPerfectoid}.
    Note that \(S_{\graded}\) has formally compatible systems of \(p\)-power roots of elements in \(\mcalS\).
    Applying \Cref{perfectoidization-graded} to this pro-\(G\)-graded semiperfectoid ring, we have the initial \(G\)-graded perfectoid ring \((S, S_{\graded})_{\perfd}\) over \((S, S_{\graded})\) such that the complete part of it is the same as \((S_{\perfd})^{\wedge I}\).
    Moreover, we have an isomorphism
    \begin{equation*}
        (S_{\perfd})^{\wedge I} \cong \parenlr{\parenlr{R[T_s^{1/p^\infty}]_{s \in \mcalS}/(f_s(T_s))_{s \in \mcalS}}_{\perfd}}^{\wedge {I}},
    \end{equation*}
    of \(S\)-algebras since both are the initial \(I\)-adic complete perfectoid \(S\)-algebras.
    By the usual Andr\'e's flatness lemma (see for example \cite[Theorem 7.14]{bhatt2022Prismsa} or \cite[Theorem VIII.3.1]{bhatt2018Geometric}), the perfectoidization \(T \defeq (R[T_s^{1/p^\infty}]_{s \in \mcalS}/(f_s(T_s))_{s \in \mcalS})_{\perfd}\) is \(p\)-completely faithfully flat over \(R\).
    
    So it suffices to show that the \(I\)-adic completion \(T^{\wedge I} = (S_{\perfd})^{\wedge I}\) is \(I\)-completely faithfully flat over \(R\).
    Since \(T\) is a perfectoid ring, \Cref{BoundedTorsionPerfd}(3) gives an isomorphism \(\dcomp{I}{T} \xrightarrow{\cong} T^{\wedge I}\).
    Then the isomorphisms
    \begin{equation*}
        T^{\wedge I} \otimes^L_{R} R/I \cong \dcomp{I}{T} \otimes^L_R R/I \cong \dcomp{I}{T \otimes^L_R R/I} \cong T \otimes^L_R R/I \cong (T \otimes^L_R R/pR) \otimes^L_{R/pR} R/I
    \end{equation*}
    holds as in the proof of \Cref{EquivCFFlat}.
    Since \(T \otimes^L_R R/pR\) is now concentrated in degree \(0\) and faithfully flat over \(R/pR\), the above isomorphisms say that \(T^{\wedge I}\) is \(I\)-completely faithfully flat over \(R\).
    
    In conclusion, \Cref{EquivCFFlat} shows that we can take the \(G\)-graded perfectoid ring \((R', R'_{\graded})\) to be \((S, S_{\graded})_{\perfd}\).
    The last assertion follows from taking transfinitely iterating the above construction as in \cite[Theorem 7.14]{bhatt2022Prismsa}.
\end{proof}

\begin{corollary}[\cref{intro:Andre-graded}] \label{Andre-graded}
Let \(R\) be a \(G\)-graded perfectoid ring and let \(\mcalS\) be a set of homogeneous elements in some graded integral extension of \(R\).
Then there exists a $p$-completely faithfully flat map $R \to R'$ of \(G\)-graded perfectoid rings such that \(R'\) admits compatible systems of \(p\)-power roots of elements in \(\mcalS\).
\end{corollary}

\begin{proof}
Since $(R^{\wedge p},R)$ is a pro-$G$-graded $p$-adic ring, the result follows from \cref{GradedAndre}.
\end{proof}

\subsection{Graded test perfectoid and its application}

We apply the existence of a graded variant of perfectoidization (\Cref{gr-perfdion}) to the perfectoid purity of graded rings.

\begin{proposition}\label{gr-test-perfd}
Let $R$ be a $G$-graded ring.
Then there exists a graded perfectoid \(R\)-algebra $R_{\infty,\gr}$ such that its \(p\)-adic completion $R_{\infty,\gr}^{\wedge p}$ is a test perfectoid over $R$ in the sense of \cite{yoshikawa2025Computation}*{Definition~2.4}.
If \(R\) is \(p\)-torsion-free, then so is \(R_{\infty, \graded}\).
\end{proposition}

\begin{proof}
By \cite{yoshikawa2025Computation}*{Theorem~A.1} and its proof, there exists a \(p\)-completely faithfully flat \(R_0\)-algebra \(R'\) such that it admits a faithfully flat morphism from a perfectoid ring \(P\) and its perfectoidization \(R_{0, \infty} \defeq R'_{\perfd}\) is a test perfectoid over \(R_0\).
Choose a subset $\mcalS$ of homogeneous elements of $R$ generating $R$ as an $R_0$-algebra such that $\mcalS \cap R_0 = \emptyset$.
Then we have a surjective graded homomorphism
\[
\varphi \colon S \defeq R_0[X_f \mid f \in \mcalS] \xrightarrow{X_f \mapsto f} R,
\]
where the grading on $S$ is given by $\deg(X_f) = \deg(f)$ and $R_0 \subseteq S_0$.

Set 
\[
S_{\infty,\gr} \defeq (S \otimes_{R_0} R_{0,\infty})[X_f^{1/p^\infty} \mid f \in \mcalS] = R_{0, \infty}[X_f^{1/p^\infty} \mid f \in \mcalS]
\quad \text{and} \quad
S_\infty \defeq S_{\infty,\gr}^{\wedge p}.
\]
Then $S_{\infty,\gr}$ is a $G[1/p]$-graded ring with $\deg(X_f^{1/p^e}) = \deg(f)/p^e$, and $ R_{0,\infty} \subseteq (S_{\infty,\gr})_0$.
In particular, $S_{\infty,\gr}$ is a graded perfectoid ring by \Cref{except-for-p-comp}(2).

We claim that $S_\infty$ is a test perfectoid over $S$.
Indeed, let $P$ be a perfectoid ring over $S$.
Replacing $P$ by a $p$-completely faithfully flat extension, we may assume that we have an $R_0$-algebra homomorphism $R_{0,\infty} \to P$ and contains compatible systems of $p$-power roots of each $X_f$ for all $f \in \mcalS$.
Then there exists an $S$-algebra homomorphism $S_{\infty,\gr} \to P$, as desired.

Next, set
\[
R_{\infty,\gr} \defeq (S_{\infty,\gr} \otimes_S R)_{\perfd,\gr},
\]
which is defined in \cref{gr-perfdion}.
Then we have $R_{\infty,\gr}^{\wedge p} = (S_\infty \otimes_S R)_{\perfd}$, which is a test perfectoid over $R$ by \cite{yoshikawa2025Computation}*{Proposition~2.6(2)}, as claimed.

Assume that \(R\) is \(p\)-torsion-free.
Since \(R_{0, \infty}\) is the perfectoidization of \(R'\), the ring \(S_{\infty}\) is the perfectoidization of \(R'[X_f^{1/p^\infty}]_{f \in \mcalS}\) and then we have
\begin{align*}
    R_{\infty, \graded}^{\wedge p} & = (R_{0, \infty}[X_f^{1/p^\infty}]_{f \in \mcalS} \otimes_S R)_{\perfd} \cong (R_{0, \infty}[X_f^{1/p^\infty}]_{f \in \mcalS} \widehat{\otimes}^L_P (P \otimes_S R))_{\perfd} \\
    & \cong (R'[X_f^{1/p^\infty}]_{f \in \mcalS})_{\perfd} \widehat{\otimes}^L_P (P \otimes_S R)_{\perfd} \cong (R'[X_f^{1/p^{\infty}}]_{f \in \mcalS} \otimes_{S} R)_{\perfd},
\end{align*}
where we use the faithfully flatness of \(R_{0, \infty}[X_f^{1/p^\infty}]_{f \in \mcalS}\) over \(P\) and the symmetric monoidal property on the perfectoidization functor in \cite[Proposition 8.13]{bhatt2022Prismsa}.
Since \(R'\) is \(p\)-completely faithfully flat over \(R_0\), so is \(R'[X_f^{1/p^\infty}]\) over \(S = R_0[X_f]\).
Since \(R\) is \(p\)-torsion-free, so is the \(p\)-completely faithfully flat base change \(R'[X_f^{1/p^{\infty}}]_{f \in \mcalS} \otimes_{S} R\).
Using \cite[Lemma A.2]{ma2022Analogue}, its perfectoidization \(R_{\infty, \graded}^{\wedge p}\) is \(p\)-torsion-free and so is the structure graded ring \(R_{\infty, \graded}\).
\end{proof}

We give two corollaries.
Recall that a Noetherian ring with \(p\) belonging the Jacobson radical is called \emph{perfectoid pure} if there exists a pure ring homomorphism \(R \to P\) to a perfectoid ring \(P\).

\begin{definition}[{cf. \cite{bhatt2024Perfectoid}}] \label{DefGradedPerfdPure}
    Let \(R\) be a \(G\)-graded Noetherian ring.
    Assume that the graded Jacobson radical of \(R\) contains \(p\).
    We say that this \(R\) is called \emph{graded perfectoid pure} if there exists a graded ring homomorphism \(R \to P\) to a graded perfectoid ring \(P\) such that it is pure as an \(R\)-module homomorphism.
\end{definition}

\begin{remark}
    Let \(R\) be a \(G\)-graded ring.
    On perfectoid purity, we remark the following.
    \begin{itemize}
        \item If \(R\) has a pure ring homomorphism \(R \to P\) to a graded perfectoid ring \(P\), then any graded maximal ideal of \(R\) automatically contains \(p\): For any graded maximal ideal \(\mfrakm\) of \(R\), the base change \(R/\mfrakm \to P \otimes_R R/\mfrakm \cong P/\mfrakm P\) is injective, in particular, \(P/\mfrakm P\) is not zero and its graded Jacobson radical contains \(p\).
        Take a graded maximal ideal of \(P/\mfrakm P\), which contains \(p\).
        Then the inverse image of this is the same as \(\mfrakm\) in \(R\) and therefore \(\mfrakm\) contains \(p\).
        \item The polynomial ring \(\mathbf{Z}_p[T]\), endowed with the \(\mathbb{Z}_{\ge 0}\)-grading \(\deg(T) = 1\), is graded perfectoid pure. However, it is not perfectoid pure, since it has a maximal ideal \((pT - 1)\) that does not contain \(p\).
    \end{itemize}
\end{remark}

\begin{corollary} \label{GradedPerfdPure}
    Let \(R\) be a \(G\)-graded Noetherian ring such that its graded Jacobson radical contains \(p\).
    Then \(R^{\wedge p}\) is perfectoid pure as a ring if and only if \(R\) is graded perfectoid pure.
\end{corollary}

\begin{proof}
    If there exists a pure graded ring homomorphism \(R \to P\) to a graded perfectoid ring \(P\).
    Since \(P^{\wedge p}\) is a perfectoid ring, then \(R^{\wedge p}\) is perfectoid pure because of \Cref{purity-graded-local}.
    
    Assume that \(R^{\wedge p}\) is perfectoid pure.
    Take a pure ring homomorphism \(R^{\wedge p} \to P\) to a perfectoid ring \(P\).
    By changing \(P\) by a \(p\)-completely faithfully flat extension, \Cref{gr-test-perfd} shows that there exists an \(R^{\wedge p}\)-algebra homomorphism \(R_{\infty, \graded}^{\wedge p} \to P\) from the \(p\)-adic completion of a graded perfectoid \(R\)-algebra \(R_{\infty, \graded}\).
    So the structure morphism \(R^{\wedge p} \to R_{\infty, \graded}^{\wedge p}\) is pure and this implies the desired purity by \Cref{purity-graded-local}.
\end{proof}

\begin{corollary}\label{gr-p-pure}
Let $R$ be a \(p\)-torsion-free $G$-graded Noetherian ring.
Assume that $R$ has a unique graded maximal ideal $\mfrakm$ with $p \in \mfrakm$.
Then $R^{\wedge p}$ is perfectoid pure if and only if so is $R_\mfrakm$.
Moreover, in this case we have $\ppt(R^{\wedge p},p) = \ppt(R_\mfrakm,p)$. 
\end{corollary}

\begin{proof}
Take a \(p\)-torsion-free graded perfectoid ring $R_{\infty,\graded}$ as in \cref{gr-test-perfd} and set its \(p\)-adic completion $R_\infty \defeq (R_{\infty,\gr})^{\wedge p}$.
Since $R_\infty$ is a \(p\)-torsion-free test perfectoid ring over $R^{\wedge p}$, we know that $R^{\wedge p}$ is perfectoid pure if and only if $R^{\wedge p} \to R_\infty$ is pure, and
\[
\ppt(R^{\wedge p},p)
= \sup\bigl\{\, \alpha \in \Z[1/p]_{\ge 0}
\ \bigm|\ 
R^{\wedge p} \to R_\infty \xrightarrow{\cdot p^\alpha} R_\infty
\text{ is pure}\,\bigr\}
\]
by \cite{yoshikawa2025Computation}*{Proposition~2.9(2)}.
Along with the \(p\)-completed localization $(R_{\infty})_{\mfrakm}^{\wedge p}$ becoming perfectoid by \cite{bhatt2019Regular}*{Example~3.8}, we can show that it becomes a \(p\)-torsion-free test perfectoid over $R_\mfrakm$.
In particular, we know that \(R_{\mfrakm}\) is perfectoid pure if and only if \(R_{\mfrakm} \to (R_{\infty})_{\mfrakm}^{\wedge p} \cong (R_{\infty, \graded})_{\mfrakm}^{\wedge p}\) is pure, and
\begin{equation*}
    \ppt(R_{\mfrakm}, p) = \sup\bigl\{\, \alpha \in \Z[1/p]_{\ge 0}\ \bigm| \ R_{\mfrakm} \to (R_\infty)_{\mfrakm}^{\wedge p} \xrightarrow{\cdot p^\alpha} (R_\infty)_{\mfrakm}^{\wedge p} \text{ is pure}\,\bigr\}
\end{equation*}
as above.
Then the desired assertion follows from applying \cref{purity-graded-local} for the graded homomorphism \(R \to R_{\infty, \graded}\).
\end{proof}

\section{Graded perfect prism and categorical equivalence}

In this section, we will give the definition of graded perfect prisms and its relation to graded perfectoid rings.

\subsection{Graded perfect prisms}

\begin{definition}\label{semi-perfectoid-prism}
We define a graded variant of perfect prisms:
\begin{itemize}
    \item Let $A$ be a $G$-graded ring.
    We say that $(A,I)$ is \emph{a graded perfect prism} if $A$ is a $\delta$-ring with $\delta(A_g) \subseteq A_{pg}$, $I$ is a homogeneous ideal, and $(A^{\wedge (p,I)},I)$ is perfect prism.
    \item \((A, A_{\graded})\) be pro-$G$-graded rings.
    We say that $(A,A_{\graded},I)$ is \emph{a pro-$G$-graded perfect prism} if $(A_{\graded},I)$ is a graded perfect prism and $(p,I)$ is topological nilpotent in $A_{\graded}$.
\end{itemize}
\end{definition}

\begin{remark}\label{rmk:first-prism}
On the definition of graded perfect prisms, we remark the following.
\begin{itemize}
    \item If $(A,A_{\graded},I)$ is a pro-$G$-graded perfect prism, then $(A, IA)$ is a perfect prism.
    \item The categorical equivalence in \cref{remk-first-pro-gra} induces the categorical equivalence
    \[
(\text{category of graded perfect prisms}) \xrightarrow{\simeq} (\text{category of pro-$G$-graded $(p,I)$-adic perfect prism})
    \]
    given by $(A,I) \mapsto (A^{\wedge (p,I)},A,I)$ and $(A,A_{\gr},I) \to (A_{\gr},I)$.
\end{itemize}
\end{remark}

\begin{proposition}\label{graded-perfect-prism}
Let \((A, A_{\graded}, I)\) be a pro-\(G\)-graded perfect prism.
Then the following hold:
\begin{enumerate}
    \item The Frobenius lift \(\phi \colon A \to A\) induces an isomorphism
    \[
        \phi_{\graded} \colon A_{\graded} \xrightarrow{\;\sim\;} A_{\graded}^{[p]}
    \]
    of \(G\)-graded topological rings.  
    In particular, \(G = G[1/p]\).

    \item There exists an element \(d \in A_0\) such that \(I = (d)\), and the pair \((A_0, (d))\) is a perfect prism.
\end{enumerate}
\end{proposition}

\begin{proof}
Let \(a \in A_{\graded}\) be a homogeneous element of degree \(g\).
Then
\[
\phi(a) = a^p + p\,\delta(a) \in A_{pg},
\]
so \(\phi\) induces a \(G\)-graded ring homomorphism
\(\phi_{\graded} \colon A_{\graded} \to A_{\graded}^{[p]}\).

\smallskip
\noindent
\textbf{Continuity of \(\phi_{\graded}\).}
Let \(I\) be an open homogeneous ideal of \(A_{\graded}\).
Then there exists \(n \ge 1\) such that \(p^n \in I\).
For any \(a_1, \ldots, a_{n+1} \in I\), we have
\begin{align*}
    \delta(a_1 \cdots a_{n+1}) 
    &\equiv p\,\delta(a_1)\delta(a_2\cdots a_{n+1}) \pmod{I} \\
    &\equiv p^2\,\delta(a_1)\delta(a_2)\delta(a_3\cdots a_{n+1}) \\
    &\equiv \cdots \equiv p^n\,\delta(a_1)\cdots\delta(a_{n+1}) \pmod{I}.
\end{align*}
Hence \(\phi(I^n) \subseteq (I, p^n) = I\), so \(\phi\) is continuous.
In particular, \(\phi\) induces a morphism of pro-\(G\)-graded rings
\[
(A, A_{\graded}) \longrightarrow (A, A_{\graded}^{[p]}).
\]

\smallskip
\noindent
\textbf{Bijectivity of \(\phi_{\graded}\).}
The injectivity of \(\phi_{\graded}\) follows from that of \(\phi\),
and its surjectivity follows from the surjectivity of \(\phi\)
together with \cref{SurjectiveProGraded}.
To show that \(\phi_{\graded}^{-1}\) is continuous,
let \(J = (x_1, \ldots, x_r)\) be a finitely generated homogeneous ideal of \(A_{\graded}\) defining the adic topology.
Then
\[
\phi_{\graded}^{-1}(J^{rp}) 
\subseteq \phi^{-1}((x_1^p, \ldots, x_r^p)) 
\subseteq (J, p),
\]
which shows that \(\phi_{\graded}^{-1}\) is continuous.
Thus \(\phi_{\graded}\) is an isomorphism of \(G\)-graded topological rings.
In particular, \(G = G[1/p]\).

\smallskip
\noindent
\textbf{Proof of (2).}
Since \(p \in I^p + \phi(I)A\), there exist \(a \in I^p\) and \(b \in \phi(I)A\) such that \(p = a + b\).
Let \(a_0\) and \(b_0\) denote the degree-zero parts of \(a\) and \(b\), respectively.
Because \(I\), \(I^p\), and \(\phi(I)A\) are homogeneous,
we have \(a_0 \in I^p\) and \(b_0 \in \phi(I)A\) by \cref{purity-homog-ideal}(1).

By the proof of \cite{bhatt2022Prismsa}*{Lemma~3.6}, we have \(\phi(I)A = (b_0)\).
Hence \(I = (d)\), where \(d \defeq \phi^{-1}(b_0) \in A_0\).
Since \(\phi_{\graded}\) induces an isomorphism
\(\phi_0 \colon A_0 \to A_0\),
the ring \(A_0\) is a perfect \(\delta\)-ring.
As \(d\) is distinguished in \(A\), it remains distinguished in \(A_0\).
Finally, because \(A_0\) is \((p,d)\)-complete,
the pair \((A_0, (d))\) forms a perfect prism.
\end{proof}

In general, the topology of the cofiltered limit of pro-\(G\)-graded rings is written by the limit topology.
However, the following proposition says that the topology on the graded ring of Witt vectors is not so ridiculous.

\begin{proposition}\label{top-prism}
Let \((A, A_{\graded}, I)\) be a \(G\)-graded perfect prism with Frobenius lift \(\phi\).
Let \(J\) be a homogeneous ideal of definition of \(A_{\graded}\), and let \(x_1, \ldots, x_r \in J\) be homogeneous generators.
By \cite{Ill79}*{Ch.~0, Section~1.3 (1.3.16)}, there exists a ring homomorphism
\[
s \colon A \longrightarrow W(A)
\]
such that the composition
\[
A \xrightarrow{s} W(A) \xrightarrow{w} \prod_{\mathbb{N}} A
\]
coincides with \((\id, \phi, \phi^2, \ldots)\), where \(w\) denotes the ghost map.  
Using this map, the following hold:
\begin{enumerate}
    \item The homomorphism \(s\) induces a \(G\)-graded ring homomorphism
    \[
    s_{\graded} \colon A_{\graded} \longrightarrow W^{\graded}(A_{\graded}),
    \]
    where \(W^{\graded}(A_{\graded})\) is defined as in \cref{graded-witt-tilt}.
    
    \item Let \(J_{\infty} \defeq ([x_1], \ldots, [x_r]) \subseteq W^{\graded}(A_{\graded}/p)\).
    Then the topology on \(W^{\graded}(A_{\graded}/p)\) coincides with the \((J_{\infty}, p)\)-adic topology.
    
    \item The composition
    \[
    A_{\graded} \xrightarrow{s_{\graded}} W^{\graded}(A_{\graded})
      \longrightarrow W^{\graded}(A_{\graded}/p)
    \]
    is denoted by \(\psi_{\graded}\).  
    Then \(\psi_{\graded}\) is an isomorphism of \(G\)-graded topological rings.
\end{enumerate}
\end{proposition}

\begin{proof}
Since \(p\) is topologically nilpotent in \(A_{\graded}\) by definition, we may assume \(p \in J\).

\smallskip
\noindent
\textbf{(1) Induced graded map.}
For every \(g \in G\), since
\[
w^{-1}(A_g \times A_{pg} \times \cdots) = W^{\graded}(A_{\graded})_g
\]
and \(\phi\) is \(G\)-graded, the homomorphism \(s\) induces a \(G\)-graded ring homomorphism
\(s_{\graded} \colon A_{\graded} \to W^{\graded}(A_{\graded})\).

\smallskip
\noindent
\textbf{(2) Description of the topology.}
To show that \(W^{\graded}(A_{\graded}/p)\) is equipped with the \((J_{\infty}, p)\)-adic topology, consider the system
\[
\{\,W^{\graded}(J^m),\, p^n \mid m,n \ge 1\,\},
\]
where
\[
W^{\graded}(J^m)
  \defeq \Ker\!\left(W^{\graded}(A_{\graded}/p)
      \longrightarrow W^{\graded}(A_{\graded}/(J^m + p))\right).
\]
Since the topology on \(A_{\graded}/p\) is the \(J\)-adic topology,
this system defines the topology on \(W^{\graded}(A_{\graded}/p)\)
by \cref{pro-graded-Witt}.
Clearly \(J_{\infty} \subseteq W^{\graded}(J)\).

Fix \(m, n \ge 1\).
Since \(J\) is generated by \(r\) elements,
\[
J^{p^{m+n}r} \subseteq (x_1^{p^{m+n}}, \ldots, x_r^{p^{m+n}}).
\]
Take any \(a \in W^{\graded}((x_1^{p^{m+n}}, \dots, x_r^{p^{m+n}}))\).
We will show that \(a \in ([x_1^{p^m}], \dots, [x_r^{p^m}], p^n)\).
Write
\[
a = (a_0, \dots, a_{n-1}) + p^n\alpha
  = [a_0] + V[a_1] + \cdots + V^{n-1}[a_{n-1}] + p^n\alpha
\]
for some \(\alpha \in W^{\graded}(A_{\graded}/p)\)
since \(A_{\graded}/p\) is perfect.
For each \(a_i\), we can write
\([a_i] = \sum_{k=1}^r [a_{i,k}][x_k^{p^{m+n}}] + V\alpha_i\)
for some \(\alpha_i \in W^{\graded}((x_1^{p^{m+n}}, \dots, x_r^{p^{m+n}}))\).
Because
\[
V^i([a_{i,k}][x_k^{p^{m+n}}])
  = V^i([a_{i,k}] \cdot F^i[x_k^{p^{m+n-i}}])
  = V^i[a_{i,k}] \cdot [x_k^{p^{m+n-i}}],
\]
we have
\[
V^i[a_i]
  = \sum_{k=1}^r V^i[a_{i,k}] \cdot [x_k^{p^{m+n-i}}]
    + V^{i+1}\alpha_i.
\]
Then \(a - \sum_{k=1}^r [a_{0,k}][x_k^{p^{m+n}}] \in VW^{\graded}((x_1^{p^{m+n}}, \dots, x_r^{p^{m+n}}))\),
so we may assume \(a_0 = 0\).
Iterating this process for \(i = 0, \dots, n-1\),
we obtain
\[
W^{\graded}((x_1^{p^{m+n}}, \ldots, x_r^{p^{m+n}}))
  \subseteq ([x_1^{p^m}], \ldots, [x_r^{p^m}], p^n)
  \subseteq (J_\infty^{p^m}, p^n).
\]
Thus \(W^{\graded}(J^{p^{m+n}r}) \subseteq (J_\infty^{p^m}, p^n)\), as claimed.

\smallskip
\noindent
\textbf{(3) The isomorphism \(\psi_{\graded}\).}
To show continuity, note that for \(a \in J\),
\[
\psi_{\graded}(a)
  = [a] + V\alpha
  = [a] + pF^{-1}\alpha
  \equiv [a] \pmod{p}
\]
for some \(\alpha \in W^{\graded}(A_{\graded}/p)\).
Hence \(\psi_{\graded}(J) \subseteq (J_\infty, p)\),
so \(\psi_{\graded}\) is continuous.

To construct the inverse, consider the map
\[
\var{w}_{\graded} \colon
W^{\graded}(A_{\graded}/p)
  \longrightarrow
A_{\graded}/p \times A_{\graded}/p^2 \times \cdots
\]
induced by the ghost map;
it is injective by \cite{Yoshikawa2025Fedder}*{Lemma~3.2}.
Write
\(\var{w}_{\graded} = (\var{w}_{0,\graded}, \var{w}_{1,\graded}, \ldots)\).
Define
\[
\varphi_{n,\graded}
  \defeq \phi^{-n} \circ \var{w}_{n,\graded}
  \colon W^{\graded}(A_{\graded}/p)
  \longrightarrow A_{\graded}/p^{n+1}.
\]
Since the ghost map is continuous, each \(\varphi_{n,\graded}\) is continuous and satisfies
\[
\varphi_{n,\graded}(\alpha)
  \equiv \varphi_{n-1,\graded}(\alpha)
  \pmod{p^n}.
\]
Indeed,
\begin{align*}
\varphi_{n,\graded}(\alpha)
  &= \phi^{-n}(a_0^{p^n} + pa_1^{p^{n-1}} + \cdots + p^n a_{n-1}) \\
  &\equiv \phi^{-(n-1)}(a_0^{p^{n-1}} + \cdots + p^{n-1} a_{n-2})
     \pmod{p^n} \\
  &= \varphi_{n-1,\graded}(\alpha).
\end{align*}
Thus these maps define
\[
\varphi_{\graded} \defeq (\varphi_{n,\graded})_n
  \colon W^{\graded}(A_{\graded}/p)
  \longrightarrow \varprojlim\nolimits_n A_{\graded}/p^n
  \simeq A_{\graded}
\]
as a morphism of \(G\)-graded topological rings.

By the definition of \(s\),
\(\varphi_{n,\graded} \circ \psi_{\graded}\colon A \to A/p^n\)
is the natural surjection,
so \(\varphi_{\graded} \circ \psi_{\graded} = \id\).
Conversely, since each \(\varphi_{n,\graded}(\alpha)\)
is the unique element mapping under \(\phi^n\)
to \(\var{w}_{n,\graded}(\alpha)\),
we have
\[
\var{w}_{n,\graded} \circ
  \psi_{\graded} \circ
  \varphi_{\graded}(\alpha)
  = a_0^{p^n} + pa_1^{p^{n-1}} + \cdots + p^n a_{n-1}
\]
for \(\alpha = (a_0, a_1, \ldots, a_{n-1})
 \in W^{\graded}(A_{\graded}/p)\).
Since \(\var{w}_{n,\graded}\) is injective,
\(\psi_{\graded} \circ \varphi_{\graded} = \id\).
Hence \(\psi_{\graded}\) is an isomorphism of \(G\)-graded topological rings.
\end{proof}

\subsection{The proof of categorical equivalence}

Finally, we establish the categorical equivalence (\Cref{CatEquivPerfdPrism}).
One of the main difficulty is that we have to consider not only the algebraic structure but also the topology on them, i.e., we want to deduce some isomorphisms on \emph{topological rings}.
Because of this, our proof is not just a mimic of the proof of the non-graded case in \cite{bhatt2022Prismsa} but more constructible proof to compare the topology.

\begin{theorem} \label{CatEquivPerfdPrism}
The following two categories are equivalent:
\begin{itemize}
    \item the category of pro-\(G\)-graded perfectoid rings \((R,R_{\graded})\);
    \item the category of pro-\(G\)-graded perfect prisms \((A,A_{\graded},I)\).
\end{itemize}
Via the forgetful functor $(R,R_{\gr}) \mapsto R$ and $(A,A_{\gr},I) \mapsto (A,I)$, the functors giving this categorical equivalence is compatible with that in \cite{bhatt2022Prismsa}*{Theorem~3.10}.    
\end{theorem}

\begin{proof}
We may assume \(G=G[1/p]\) by \cref{graded-perfectoid-equiv,graded-perfect-prism}.
The last assertion follows from the proof below.
Let \(\mathcal{C}_1\) denote the category of \(G\)-graded perfectoid rings \((R,R_{\graded})\), and \(\mathcal{C}_2\) the category of \(G\)-graded perfect prisms \((A,A_{\graded},I)\).

\smallskip
\noindent\textbf{Step 1. Constructing the functor \(\mathcal{F}\colon \mathcal{C}_2\to\mathcal{C}_1\).}
Let \((A,A_{\graded},I)\) be a \(G\)-graded perfect prism.
By \cref{graded-perfect-prism}, there exists \(d\in A_0\) with \(I=(d)\).
Since \(I=dA\cong A\), the \(A\)-module \(A/dA\) is derived complete.  
Using \citeSta{0G3I}, the closure \(\overline{dA}\) in \(A\) is nilpotent in \(A/dA\).  
As \(A/dA\) is perfectoid (\Cref{rmk:first-prism}) and hence reduced, we get \(\overline{dA}=dA\).
By \cref{purity-homog-ideal},
\[
\overline{dA_{\graded}}
= \overline{dA}\cap A_{\graded}
= dA\cap A_{\graded}
= dA_{\graded}.
\]
Hence, by \cref{quot-pro-gra}, the pair \((A/dA,A_{\graded}/dA)\) is a pro-\(G\)-graded ring with the quotient topology.  
Since \(A/dA\) is perfectoid as noted above, the pro-\(G\)-graded ring \((A/dA,A_{\graded}/(d))\) is perfectoid.
We define
\[
\mathcal{F}(A,A_{\graded},I)\defeq(A/dA,A_{\graded}/dA_{\graded)}.
\]

\smallskip
\noindent\textbf{Step 2. Constructing the functor \(\mathcal{G}\colon \mathcal{C}_1\to\mathcal{C}_2\).}
Let \((R,R_{\graded})\) be a \(G\)-graded perfectoid ring.  
Choose a finitely generated homogeneous ideal \(J\subset R_{\graded}\) defining the topology, with \(p\in J\).
Because \(R_0\) is perfectoid, there exists \(\varpi\in R_0\) with a compatible system \(\{\varpi^{1/p^e}\}_e\) and \(x\varpi^p=p\) for some \(x\in R_0^\times\).

By \cref{pro-graded-tilt}, the pair \((R^\flat,R^{\flat,\graded}_{\graded})\) is a pro-\(G\)-graded ring, where \(R^{\flat,\graded}_{\graded}\) is as in \cref{graded-witt-tilt}(1).  
Let \(\pi_e\colon R^{\flat,\graded}_{\graded}\to (R_{\graded}/pR_{\graded})^{[p^{-e}]}\) and \(\pi'_e\colon R^\flat\to (R/pR)_e\) be the canonical projections; these define the limit topologies on \(R^\flat\) and \(R^{\flat,\graded}_{\graded}\).  
Arguing as in the standard proof for perfectoid rings, one obtains
\[
\Ker(\pi_e)
  =(\varpi^\flat)^{p^{e+1}}R^\flat_{\graded}.
\]
Pulling back \(J/pR_{\graded}\) along \(\pi_0\) gives the \(J^\flat\)-adic topology on \(R^{\flat,\graded}_{\graded}\), where \(J^\flat\supset (\varpi^\flat)^p\) and \(\pi_0(J^\flat)=J\bmod p\).  
Since the map \(\sharp_{0,\graded}\colon R^{\flat,\graded}_{\graded}\to R_{\graded}/pR_{\graded}\) is surjective (by \cref{graded-perfectoid-equiv}), \(J^\flat\) is finitely generated by homogeneous lifts of generators of \(J\) modulo \((\varpi^\flat)^p\).

Now define
\[
(A,A_{\graded})
  \defeq (A_{\inf}(R),A_{\inf}^{\graded}(R_{\graded}))
  = (W(R^\flat),W^{\graded}(R^{\flat,\graded}_{\graded})).
\]
This is a pro-\(G\)-graded ring by \cref{pro-graded-Witt} and \cref{SharpMapThetaMap}, and its graded structure is
\[
A_g
  = \{(a_0,a_1,\ldots)\in W^{\graded}(R^{\flat,\graded}_{\graded})
     \mid a_i\in R^\flat_{p^i g}\ \forall i\ge0\}.
\]

Let \(\theta_{\graded}\colon A_{\graded}\to R_{\graded}\) be the graded theta map from \cref{SharpMapThetaMap}.  
It is surjective because each finite-level map \(\theta_{n,\graded}\colon W_n(R^{\flat,\graded}_{\graded})\to R_{\graded}/p^n\) is surjective by induction on \(n\).  
On degree \(0\), \(\theta_0\colon W(R_0^\flat)\to R_0\) is Fontaine’s theta map, which is surjective since \(R_0\) is perfectoid.  
Choose a lift \(\widetilde{x}\in A_0\) of \(x\in R_0^\times\) with \(\theta_0(\widetilde{x})=x\), and define
\[
d\defeq p-[\varpi^\flat]^p\,\widetilde{x}\in A_0.
\]
By \cite{bhatt2018Integral}*{Lemma~3.10}, we have \(\Ker(\theta)=dA\) and \(\Ker(\theta_0)=dA_0\).

The topology on \(A_{\graded}\) is generated by the ideals
\[
\{\,W^{\graded}((J^\flat)^m),\ p^n \mid m,n\ge 1\,\},
\]
and since \([\varpi^\flat]\) is topologically nilpotent, so is \(d\).  
If \(x_1,\ldots,x_r\) generate \(J^\flat\) homogeneously, then by \cref{top-prism}(3) we have an isomorphism
\[
A_{\graded}\xrightarrow{\ \psi_{\graded}\ } W^{\graded}(A_{\graded}/p)
  \cong W^{\graded}(R^{\flat,\graded}_{\graded}),
\]
whose target has the \((J_\infty,p)\)-adic topology with
\(J_\infty\defeq([x_1],\ldots,[x_r])\).  
Hence \(A_{\graded}\) is a \(G\)-graded adic ring.  
The Frobenius lift
\[
\phi\colon A=W(R^\flat)\to A,\qquad (a_0,a_1,\ldots)\mapsto(a_0^p,a_1^p,\ldots),
\]
preserves the grading and induces \(\phi_{\graded}\colon A_{\graded}\to A_{\graded}^{[p]}\).  
Therefore the triple \((A,A_{\graded},(d))\) is a \(G\)-graded perfect prism, and we define
\[
\mathcal{G}(R,R_{\graded})\defeq(A,A_{\graded},(d)).
\]

\smallskip
\noindent\textbf{Step 3. Showing \(\mathcal{F}\circ\mathcal{G}\simeq\id\).}
For \((R,R_{\graded})\), let \(\mathcal{G}(R,R_{\graded})=(A,A_{\graded},(d))\).  
Since \(\Ker(\theta)=dA\), by \cref{purity-homog-ideal} we also have \(\Ker(\theta_{\graded})=dA_{\graded}\).  
Thus it suffices to show that \(\theta_{\graded}\colon A_{\graded}/(d)\to R_{\graded}\) is a homeomorphism.  
By construction, \(\theta_{\graded}(J_\infty)\subseteq(J,p)\), so \(\theta_{\graded}\) is continuous.  
Let \(y_1,\ldots,y_s\) be homogeneous generators of \(J\).  
Since \(x_i=\sharp_{\graded}(x_i^\flat)\) lift generators of \(J^\flat\), we have \(y_i^{p^n}=\theta_{\graded}([x_i])\).  
Hence \(J^{sp^n}\subseteq(y_1^{p^n},\ldots,y_s^{p^n})\subseteq(\theta_{\graded}(W^{\graded}(J^\flat)),p^n)\) for all \(n\), so \(\theta_{\graded}\) is open.  
Therefore \(A_{\graded}/(d)\xrightarrow{\sim}R_{\graded}\) as topological \(G\)-graded rings, proving \(\mathcal{F}\circ\mathcal{G}\simeq\id\).

\smallskip
\noindent\textbf{Step 4. Showing \(\mathcal{G}\circ\mathcal{F}\simeq\id\).}
Let \((A,A_{\graded},I)\) be a \(G\)-graded perfect prism and set \(\mathcal{F}(A,A_{\graded},I)=(R,R_{\graded})=(A/(d),A_{\graded}/(d))\).  
Let \(\mathcal{G}(R,R_{\graded})=(A',A'_{\graded},(d'))\).  
We first show \((A'/(p),A'_{\graded}/(p))\simeq(A/(p),A_{\graded}/(p))\).
Consider
\[
\sigma\colon A/(p)\cong(\varprojlim_{\phi}A)/p
  \longrightarrow\varprojlim_{F}(A/(p,d))
  \cong R^\flat=A'/(p),
\]
and the induced \(G\)-graded map \(\sigma_{\graded}\colon A_{\graded}/(p)\to A'_{\graded}/(p)\).
Take a finitely generated homogeneous ideal \(J\subset A_{\graded}\) defining its topology, with \(p,d\in J\).
Let \(J_R\defeq JR\) and \(J^\flat\) be the unique ideal of \(R^\flat\) with \((\varpi^\flat)^p\subset J^\flat\) and \(J^\flat R^\flat/(\varpi^\flat)^p=JR/(p)\).
Then \(\sigma(JA/(p))=J^\flat\), so \(\sigma_{\graded}\) is continuous and open.
By construction, the induced map \(A_{\graded}/(d,p)\to A'_{\graded}/(p,d')\) is an isomorphism, hence so is \(A_{\graded}/(p,d^n)\to A'_{\graded}/(p,(d')^n)\) by induction.
By \cref{graded-top-nakayama}, \(\sigma\) is an isomorphism, and thus
\[
(\sigma,\sigma_{\graded})\colon
(A/(p),A_{\graded}/(p))
  \xrightarrow{\ \sim\ }
(A'/(p),A'_{\graded}/(p)).
\]
Hence
\[
A_{\graded}\xrightarrow{\ \psi_{\graded}\ }
  W^{\graded}(A_{\graded}/p)
  \xrightarrow{\,W^{\graded}(\sigma_{\graded})\,}
  W^{\graded}(A'_{\graded}/p)
  =W^{\graded}(R^\flat_{\graded})
  =A'_{\graded}
\]
is an isomorphism of \(G\)-graded topological rings (with \(\psi_{\graded}\) as in \cref{top-prism}(3)).
This induces an isomorphism \((A,A_{\graded},(d))\xrightarrow{\sim}(A',A'_{\graded},(d'))\) of \(G\)-graded perfect prisms.  
Hence \(\mathcal{G}\circ\mathcal{F}\simeq\id\).

\smallskip
\noindent\textbf{Conclusion.}
The functors \(\mathcal{F}\) and \(\mathcal{G}\) are quasi-inverse equivalences between the two categories.
\end{proof}


\begin{corollary}[\Cref{intro:cat-equiv}] \label{cat-equiv}
The following two categories are equivalent:
\begin{itemize}
    \item the category of \(G\)-graded perfectoid rings $R$;
    \item the category of \(G\)-graded perfect prisms \((A,I)\).
\end{itemize}
This categorical equivalence is compatible with the usual one in \cite[Theorem 3.3]{bhatt2022Prismsa} through the completion functors \(R \mapsto R^{\wedge p}\) and \((A, I) \mapsto (A^{\wedge(p, I)}, I)\).
\end{corollary}

\begin{proof}
It follows from \cref{first-cat-equiv,rmk:first-prism,CatEquivPerfdPrism}.
\end{proof}


\end{document}